\documentclass[final,leqno]{siamart171218}%
\usepackage{amsmath,amssymb,amsfonts}
\usepackage{graphicx}
\usepackage[caption=false]{subfig}
\usepackage{float}
\usepackage{bm}
\usepackage{color}
\usepackage{url,hyperref}
\usepackage{appendix}
\usepackage[notcite,notref]{showkeys}
\usepackage{version,mathrsfs}
\usepackage{algorithmic}
\usepackage{mathtools}
\usepackage{amsmath}
\usepackage{amsfonts}
\usepackage{amssymb}%
\Crefname{ALC@unique}{Line}{Lines}
\allowdisplaybreaks
\newcommand{\bd}[1]{\mathbf{#1}}

\def \i {{\mathrm{i}}}

\newsiamremark{remark}{Remark}
\newsiamremark{assumption}{Assumption}
\begin{document}

\title{Exponential convergence for multipole and local expansions and their
translations for sources in layered media: 2-D acoustic wave\thanks{This work
was supported by US Army Research Office (Grant No.W911NF-17-1-0368) and US
National Science Foundation (Grant No. DMS-1802143).}}
\author{Wenzhong Zhang\thanks{Department of Mathematics, Southern Methodist
University, Dallas, TX 75275.}
\and Bo Wang\thanks{LCSM, Ministry of Education, School of Mathematics and
Statistics, Hunan Normal University, Changsha, Hunan 410081, P. R. China.
Department of Mathematics, Southern Methodist University, Dallas, TX 75275.
The author acknowledges the financial support provided by NSFC (grant
11771137), the Construct Program of the Key Discipline in Hunan Province and a
Scientific Research Fund of Hunan Provincial Education Department (No.
16B154).}
\and Wei Cai\thanks{Corresponding author, Department of Mathematics, Southern
Methodist University, Dallas, TX 75275 (\texttt{cai@smu.edu}).
\textit{Submitted to SIAM J. Numerical Analysis, June, 2019 and a first version of this paper
appeared as arXiv:1809.07716 on September 20, 2018.}}}
\maketitle

\begin{abstract}
In this paper, we will first give a derivation of the multipole expansion (ME)
and local expansion (LE) for the far field from sources in general 2-D layered
media and the multipole-to-local translation (M2L) operator by using the
generating function for Bessel functions. Then, we present a rigorous proof of
the exponential convergence of the ME, LE, and M2L for 2-D Helmholtz equations
in layered media. It is shown that the convergence of ME, LE, and M2L for the
reaction field component of the Green's function depends on a polarized distance
between the target and a polarized image of the source.

\end{abstract}




\section{Introduction}

The multipole expansion (ME), local expansion (LE), and multipole-to-local
translation (M2L) form the mathematical foundation of fast multipole methods
(FMMs) for evaluating integral operators defined by the Green's function of
Helmholtz equations in wave scattering \cite{rokhlin1990rapid}. The ME for the
Green's functions in the free space was based on the Graf's addition theorems
for Bessel functions. To extend the FMM for wave scattering in layered media,
the ME and M2L formulas for Helmholtz equations in a 2-D half-space domain were
proposed
in \cite{cho-arxiv}.
The derivation in \cite{cho-arxiv} for the ME and M2L for the Green's function
in a half-space domain with impedance boundary condition made use of the image
charge (point and line images) representation of the Green's function of the
domain and the MEs, based on the Graf's addition theorem, for the image charges
as well as the original source charges. And, it was shown that the ME
coefficients used to compress the far field of the source charges in the free
space can also be used to compress the far field of the images, therefore,
producing a ME for the Green's functions of the 2-D half-space domain. It was
predicted in \cite{cho-arxiv} that similar results could hold for general
layered media by using a Sommerfeld representation in frequency domain of the
Green's function \cite{chew} where the layered Green's function is expressed
as plane waves in the frequency domain. Furthermore, in the case of the half
space with an impedance boundary condition, the image representation of the
domain Green's function justifies the same truncation order, thus the
exponential convergence, of ME and M2L and a heterogeneous FMM for sources in
the half space domain was proposed and implemented in \cite{cho2018}, giving
an $O(N)$ complexity of evaluating the integral operator of low frequency
Helmholtz operators.

As an image representation of general layered media Green's function may not
exist, in this paper, we will present an alternative complete derivation for
the ME, LE, and M2L operators for the Green's function in general 2-D layered
media by using the generating function for the Bessel functions of the first
kind (referred as the Bessel generating function in this paper). Also, we will
give a rigorous proof of the exponential convergence of the ME, LE, and M2L
translation operators for acoustic wave sources in general 2-D layered media.
The convergence analysis reveals a fact that the convergence of ME, LE,
and M2L for the reaction field component of the Green's function in fact depends on
a polarized distance, which is defined between the target and a polarized
image of the source, thus suggesting how the FMM framework should be set for
sources and targets in layered media.

The rest of the paper is organized as follows. In \Cref{sect-2}, we first give
some technical tools crucial to the work in this paper, including the Bessel
generating function, which relates plane waves to cylindrical waves and the
growth condition of the Bessel functions. Then, the Bessel generating function
is used to derive the analytical formula for the ME expansions for sources in
2-D layered media, the M2L translation operator, and the multipole and local
translation operators. The exponential convergence rates for these expansions
are then given. In \Cref{sect-3}, we will first give the proof of the
exponential convergence of some integral expansions resulting from using the Bessel
generating function. The proof is given starting with a special case
corresponding to the situation when the far-field location is directly above
or below the center of the expansion. Then, the Cagniard--de Hoop transform \cite{chew} is
introduced so that we can deal with the general case by using complex domain
contour integrals. The proof for the error estimate of ME, M2L, etc.
introduced in \Cref{sect-2} will follow. A conclusion is given in
\Cref{sect-4} while Appendices are included for some technical lemmas and
proofs of several lemmas from the main text.

\section{Far-field expansions for the 2-D Helmholtz equation in layered media}

\label{sect-2} In this section, we begin with some properties of the Bessel
functions of the first kind, which inspires an alternative derivation of the
ME of the free space Green's function. These properties will be key to derive
various far-field expansions in layered media. The ME, LE, M2L, and the local
to local translation (L2L) for the layered media will then be derived with error
estimates. Finally, a feasible FMM framework for sources in layered media is
proposed based on the convergence results of the far-field expansions.

\subsection{An identity and some estimates on Bessel functions of the first
kind}

Recall the Bessel generating function \cite[(9.1.41)]{abramowitz1966handbook},
for any $z,\omega\in\mathbb{C}$ with $\omega\neq0$,
\begin{equation}
g(z,\omega)=\exp\left(  \frac{z}{2}(\omega-\omega^{-1})\right)  =\sum
_{p=-\infty}^{\infty}J_{p}(z)\omega^{p}.\label{eq-genBessel}%
\end{equation}
The identity \cref{eq-genBessel} expresses a plane wave function in terms of
cylindrical functions,  in contrast to the Sommerfeld integral representation
of the Green's function (cylindrical function) in terms of plane waves
(\ref{sommerfeld}). This duality facilitates the derivation of the far-field
expansions in this paper.

The above series converges absolutely, which is a corollary of the following
lemma.
\begin{lemma}[an estimate on Bessel functions of the first kind]\label{lemma-Jn-bound}Let $p\in\mathbb{Z}$, $z\in\mathbb{C}$, $p$ and $z$ are not both zero.
Then
\begin{equation*}
\left|J_{p}(z)\right| \le \frac{1}{|p|!} \left(\frac{|z|}{2} \right)^{|p|} e^{|\Im z|}.
\end{equation*}
\end{lemma}\begin{proof}
When $p \ge 0 > -\frac{1}{2}$, the inequality is exactly given by \cite[(9.1.62)]{abramowitz1966handbook}.
Then, the identity $J_{p}(z) = (-1)^{p}J_{-p}(z)$ covers the case $p<0$.
\end{proof}In particular, for $z\in\mathbb{R}$ and $z\geq0$, the inequality
\begin{equation}
\left\vert J_{p}(z)\right\vert \leq\frac{1}{|p|!}\left(  \frac{z}{2}\right)
^{|p|}\label{eq-BesselBound}%
\end{equation}
(with the convention $0^{0}=1$) will be used to derive the exponential
convergence estimates for far-field expansions in this paper.

\subsection{The multipole expansion in free space revisited}

\label{section-free-space}Consider $N$ sources with strength $q_{j}$ placed at
locations $\mathbf{x}_{j}=(x_{j},y_{j})$, $j=1,2,\cdots,N$ within a circle
centered at $\mathbf{x}_{c}=(x_{c},y_{c})$ with a radius $r$ in the free space
$\mathbb{R}^{2}$, then, the field located at $\mathbf{x}$ due to all sources
is given by
\[
u^{\text{f}}(\mathbf{x})=\sum_{j=1}^{N}q_{j}G^{\text{f}}(\mathbf{x}%
,\mathbf{x}_{j}),
\]
where $G^{\text{f}}$ is the free space Green's function
\[
G^{\text{f}}(\mathbf{x},\mathbf{x}^{\prime})=\frac{\i }{4}H_{0}%
^{(1)}\left(  k|\mathbf{x}-\mathbf{x}^{\prime}|\right)  ,
\]
$k$ is the wave number, and $H_{0}^{(1)}$ is the Hankel function of the first
kind. A target $\mathbf{x}$ is well-separated from the sources if the
distance between $\mathbf{x}$ and the source center $\mathbf{x}_{c}$
is at least $2r$.

By using Graf's addition theorem \cite{abramowitz1966handbook}, the free space
Green's function for the well-separated sources $\mathbf{x}_{j}$ and the
target $\mathbf{x}$ can be compressed as a multipole expansion given by
\begin{equation}
u^{\text{f}}(\mathbf{x})=\frac{\i }{4}\sum_{p=-\infty}^{\infty}\alpha
_{p}H_{p}^{(1)}\left(  {k}|\mathbf{x}-\mathbf{x}_{c}|\right)  e^{\i
p\theta_{c}}\approx\frac{\i }{4}\sum_{|p|<P}\alpha_{p}H_{p}^{(1)}({k}%
\rho_{c})e^{\i  p\theta_{c}},\label{eq-free}%
\end{equation}
where $\alpha_{p}=\sum_{j=1}^{N}q_{j}J_{p}({k}\rho_{j})e^{-\i  p\theta_{j}%
}$, $(\rho_{c},\theta_{c})$ are the polar coordinates of $\mathbf{x}%
-\mathbf{x}_{c}$, $(\rho_{j},\theta_{j})$ are the polar coordinates of
$\mathbf{x}_{j}-\mathbf{x}_{c}$, and the truncation index $P$ is a constant
independent of the number of the sources $N$ \cite{rokhlin1990rapid}.

The multipole expansion can also be derived in the frequency domain using \cref{eq-genBessel} as
follows. Consider one source $\mathbf{x}_{j}$ and suppose $y-y_{j}>0$,
$y-y_{c}>0$ for simplicity. The interaction between $\mathbf{x}$ and
$\mathbf{x}_{j}$ can be represented by a Sommerfeld integral of plane waves \cite{cho2018},
\begin{equation}
G^{\text{f}}(\mathbf{x},\mathbf{x}_{j})=\frac{\i }{4}H_{0}^{(1)}\left(
k|\mathbf{x}-\mathbf{x}_{j}|\right)  =\frac{\i }{4}\frac{1}{\i \pi}%
\int_{-\infty}^{\infty}\frac{e^{-\sqrt{\lambda^{2}-k^{2}}(y-y_{j})}}%
{\sqrt{\lambda^{2}-k^{2}}}e^{\i \lambda(x-x_{j})}d\lambda
,\label{sommerfeld}%
\end{equation}
while each term $H_{p}^{(1)}(k\rho_{c})e^{\i  p\theta_{c}}$ in
\cref{eq-free} has a similar representation \cite{cho2018}
\begin{equation}
H_{p}^{(1)}(k\rho_{c})e^{\i  p\theta_{c}}=\frac{1}{\i \pi}\int_{-\infty
}^{\infty}\frac{e^{-\sqrt{\lambda^{2}-k^{2}}(y-y_{c})}}{\sqrt{\lambda
^{2}-k^{2}}}e^{\i \lambda(x-x_{c})}(-\i )^{p}\left(  \frac{\lambda
-\sqrt{\lambda^{2}-k^{2}}}{k}\right)  ^{p}d\lambda,\label{eq-PWE-free}%
\end{equation}
here the square root in $\sqrt{\lambda^{2}-k^{2}}$ for $|\lambda|<k$ is
defined as $\sqrt{\lambda^{2}-k^{2}}=-\i \sqrt{k^{2}-\lambda^{2}}$. These
integral forms give an alternative derivation for the multipole expansion of
$\frac{\i }{4}H_{0}^{(1)}(k|\mathbf{x}-\mathbf{x}_{j}|)=\frac{\i }%
{4}H_{0}^{(1)}(k|\left(  \mathbf{x}-\mathbf{x}_{c}\right)  +\left(
\mathbf{x}_{c}-\mathbf{x}_{j}\right)  |)$ with separable product terms
involving $\left(  \mathbf{x}-\mathbf{x}_{c}\right)  $ and $\left(
\mathbf{x}_{c}-\mathbf{x}_{j}\right)  $,
\begin{align*}
&  \frac{\i }{4}H_{0}^{(1)}\left(  k|\mathbf{x}-\mathbf{x}_{j}|\right)
=\frac{\i }{4}\frac{1}{\i \pi}\int_{-\infty}^{\infty}\frac
{e^{-\sqrt{\lambda^{2}-k^{2}}(y-y_{j})}}{\sqrt{\lambda^{2}-k^{2}}}%
e^{\i \lambda(x-x_{j})}d\lambda\\
={} &  \frac{\i }{4}\frac{1}{\i \pi}\int_{-\infty}^{\infty}%
\frac{e^{-\sqrt{\lambda^{2}-k^{2}}(y-y_{c})}}{\sqrt{\lambda^{2}-k^{2}}%
}e^{\i \lambda(x-x_{c})}\cdot e^{-\sqrt{\lambda^{2}-k^{2}}(y_{c}%
-y_{j})+\i \lambda(x_{c}-x_{j})}d\lambda\\
={} &  \frac{\i }{4}\frac{1}{\i \pi}\int_{-\infty}^{\infty}%
\frac{e^{-\sqrt{\lambda^{2}-k^{2}}(y-y_{c})}}{\sqrt{\lambda^{2}-k^{2}}%
}e^{\i \lambda(x-x_{c})}\cdot g\left(  k\rho_{j},-\i  e^{-\i
\theta_{j}}w(\lambda)\right)  d\lambda\\
={} &  \frac{\i }{4}\frac{1}{\i \pi}\int_{-\infty}^{\infty}%
\frac{e^{-\sqrt{\lambda^{2}-k^{2}}(y-y_{c})}}{\sqrt{\lambda^{2}-k^{2}}%
}e^{\i \lambda(x-x_{c})}\cdot\sum_{p=-\infty}^{\infty}J_{p}(k\rho
_{j})e^{-\i  p\theta_{j}}\left(  -\i  w(\lambda)\right)  ^{p}d\lambda\\
={} &  \frac{\i }{4}\sum_{p=-\infty}^{\infty}J_{p}(k\rho_{j})e^{-\i
p\theta_{j}}\cdot\frac{1}{\i \pi}\int_{-\infty}^{\infty}\frac
{e^{-\sqrt{\lambda^{2}-k^{2}}(y-y_{c})}}{\sqrt{\lambda^{2}-k^{2}}}%
e^{\i \lambda(x-x_{c})}\left(  -\i  w(\lambda)\right)  ^{p}d\lambda\\
={} &  \frac{\i }{4}\sum_{p=-\infty}^{\infty}J_{p}(k\rho_{j})e^{-\i
p\theta_{j}}\cdot H_{p}^{(1)}(k\rho_{c})e^{\i  p\theta_{c}},
\end{align*}
here
\begin{equation}
w(\lambda)=\frac{\lambda-\sqrt{\lambda^{2}-k^{2}}}{k}.
\end{equation}
The interchangeability of the sum and the integration is verified by the
validity of the identity itself, i.e. the Graf's addition theorem.

\subsection{The Green's function in layered media}

Consider a horizontally layered medium with $L$ interfaces located at
$y=d_{l}$, $0\leq l\leq L-1$, arranged from top to bottom as $l$ increases.
Each interface $y=d_{l}$ separates layer $l$ above layer $l+1$. Each layer $l$
is homogeneous with a wave number $k_{l}>0$, $0\leq l\leq L$.

We assume $s$ labels the layer where the source $\mathbf{x}^{\prime
}=(x^{\prime},y^{\prime})$ locates, and $t$ the layer where the target
$\mathbf{x}=(x,y)$ locates, $0\leq s,t\leq L$.

The layered Green's function $G(\mathbf{x},\mathbf{x}^{\prime})$ for the
Helmholtz equation is a piecewise function for source $\mathbf{x}^{\prime}$
and target $\mathbf{x}$ from possibly different layers. Within each layer,
\begin{equation}
\Delta G(\mathbf{x},\mathbf{x}^{\prime})+k_{t}^{2}G(\mathbf{x},\mathbf{x}%
^{\prime})=-\delta(\mathbf{x},\mathbf{x}^{\prime}),\label{eq-layerHelmholtz}%
\end{equation}
with two interface conditions at $y=d_{l}$ of the form
\begin{equation}
\left[  a_{t}G+b_{t}\frac{\partial G}{\partial\mathbf{n}}\right]
=0,\label{eq-interface-cond-form}%
\end{equation}
where the bracket $[\cdot]$ refers to the jump of the quantity inside at the
interface, and $a_{t}$ and $b_{t}$ are some complex numbers (depending on the
layer number $t$).

Note that the right-hand side of equation \cref{eq-layerHelmholtz} is nonzero
only when $\mathbf{x}$ and $\mathbf{x}^{\prime}$ are in the same layer, i.e.
$s=t$. Define
\begin{equation}
\label{eq-reaction-field}u^{\text{r}}(\mathbf{x},\mathbf{x}^{\prime
})=G(\mathbf{x},\mathbf{x}^{\prime})-\delta_{t,s}G_{s}^{\text{f}}%
(\mathbf{x},\mathbf{x}^{\prime}),
\end{equation}
here $\delta_{t,s}$ is the Kronecker delta function, $G_{s}^{\text{f}%
}(\mathbf{x},\mathbf{x}^{\prime})=\frac{\i}{4}H_{0}^{(1)}\left( k_{s}%
|\mathbf{x}-\mathbf{x}^{\prime}|\right)  $ is the free-space Green's function
with wave number $k_{s}$. $u^{\text{r}}$ is called the reaction field using the
terminology of electrostatics \cite{cai2013}, and satisfies an homogeneous Helmholtz equation
within each layer. We have the following proposition for the reaction field
$u^{\text{r}}$ and the details are given in \cite{bo2018taylorfmm}.
\begin{proposition}[decomposition of the reaction field]\label{lemma-decomp-rf}
Suppose the Helmholtz problem in layered media is well-posed.
Then, the reaction field $u^{\mathrm{r}}$ can be decomposed into the following sum
\begin{equation}\label{eq-urf-decomp}
\begin{aligned}
u^{\mathrm{r}}(\bd{x},\bd{x}')= {}& u_{ts}^{\uparrow\uparrow}\left( \bd{x}, \bd{x}'; \sigma_{ts}^{\uparrow \uparrow} \right)+u_{ts}^{\uparrow\downarrow}\left( \bd{x}, \bd{x}'; \sigma_{ts}^{\uparrow \downarrow} \right) \\
&+u_{ts}^{\downarrow\uparrow}\left( \bd{x}, \bd{x}'; \sigma_{ts}^{\downarrow \uparrow} \right)+u_{ts}^{\downarrow\downarrow}\left( \bd{x}, \bd{x}'; \sigma_{ts}^{\downarrow \downarrow} \right)=\sum_{\ast \star}u_{ts}^{\ast\star}\left( \bd{x}, \bd{x}'; \sigma_{ts}^{\ast \star} \right)
\end{aligned}
\end{equation}
of integrals
\begin{equation}\label{eq-urf-decomp-int}
u_{ts}^{\ast\star}\left( \bd{x}, \bd{x}'; \sigma_{ts}^{\ast \star} \right)=\int_{-\infty}^{\infty} \mathcal{E}_{ts}^{\ast\star}(\bd{x},\bd{x}',\lambda) \sigma_{ts}^{\ast \star}(\lambda)d\lambda
\end{equation}
with an exponential factor
\begin{equation}\label{eq-E-urf-decomp-expo}
\mathcal{E}_{ts}^{\ast\star}(\bd{x},\bd{x}',\lambda)=e^{-\sqrt{\lambda^2-k_t^2}\tau^{\ast}(y-d_t^{\ast})-\sqrt{\lambda^2-k_s^2}\tau^{\star}(y'-d_s^{\star})+\i\lambda(x-x')}
\end{equation}
and a coefficient term $\sigma_{ts}^{\ast \star}(\lambda)$ which does not depend on the coordinates of $\bd{x}$ and $\bd{x}'$ in the integrand.
Here, $\ast, \star$ ranged in $\{ \uparrow, \downarrow \}$ mark the vertical field propagation directions corresponding to the target and the source, respectively, while the incoming options from $y=\pm\infty$ are prohibited from the sum in \cref{eq-urf-decomp} (for instance, if $\bd{x}$ and $\bd{x}'$ are both in the top layer, then \cref{eq-urf-decomp} becomes $u^{\mathrm{r}}=u_{00}^{\uparrow\uparrow}$).
The interfaces $d_l^\uparrow = d_l$ for $l\ne L$, and $d_l^\downarrow = d_{l-1}$ for $l \ne 0$.
$\tau^{\uparrow} = 1$, $\tau^{\downarrow} = -1$, together they will guarantee $\tau^{\ast}(y-d_t^{\ast})>0$ and $\tau^{\star}(y'-d_s^{\star})>0$.
\end{proposition}

\begin{remark}
The specific form of the exponential term $\mathcal{E}_{ts}^{\ast\star}(\bd{x},\bd{x}',\lambda)$ is introduced to ensure that each
coefficient term $\sigma_{ts}^{\ast\star}(\lambda)$ will have a polynomial
growth rate under certain conditions, to be elaborated in
\Cref{Appendix-sigma}. The polynomial growth of $\sigma_{ts}%
^{\ast\star}(\lambda)$ will be needed for the exponential convergence estimate of ME,
LE, M2L, and L2L expansions. This specific form also results in a dependence
of the exponential convergence on a special ``polarization distance" between a
source and a target in the layered media, as defined in \cref{poldist} and depicted in \Cref{fig:picme}.
\end{remark}

The integrand of \cref{eq-urf-decomp-int} may have real poles which cause
difficulty when being integrated.
However, such integrals should be treated as
the limiting case of the field in lossy physical media.
To understand the real
poles in the integrand, we first introduce the necessary branch cut of the
square roots. For any $z=re^{\i \theta}\in\mathbb{C}$ with $r\geq0$,
$\theta\in\lbrack-\pi,\pi)$, define
\begin{equation}
\sqrt{z}=\sqrt{r}e^{\i \frac{\theta}{2}}.\label{eq-branch-cut}%
\end{equation}
For each square root $\sqrt{\lambda^{2}-k_{l}^{2}}$, the corresponding branch
cut in the $\lambda$-plane is the union of the imaginary axis and the real
interval $[-k_{l},k_{l}]$. In a realistic physical case where the medium in
layer $l$ is lossy with a perturbed wave number $\tilde{k}_{l}=k_{l}%
+\epsilon_{l}\i $, $\epsilon_{l}>0$, the perturbed branch cut is then shown
in \cref{fig:picctrsmf}. The branch cut of $\sqrt{\lambda^{2}-k_{l}^{2}}$ is
the limit of the perturbed one as $\epsilon_{l}\rightarrow0^{+}$.
\begin{figure}[tbh]
\centering \includegraphics[scale=0.5]{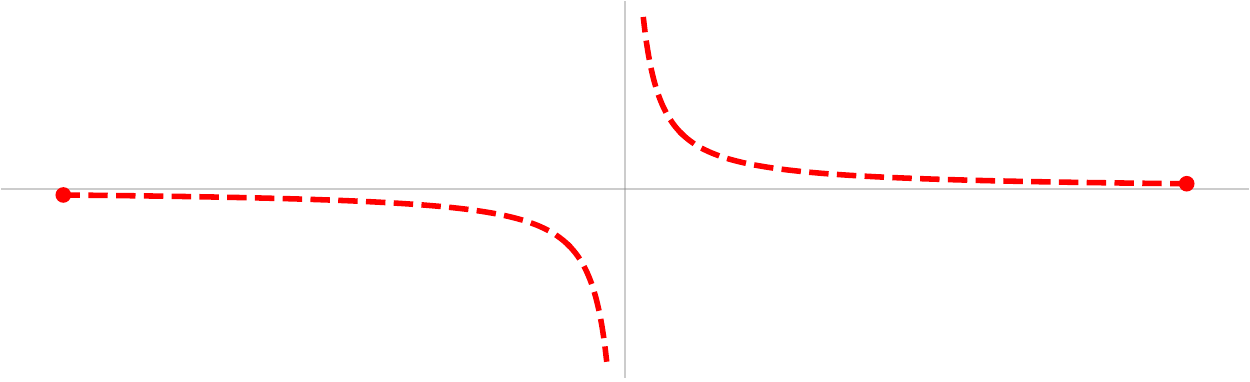}\caption{The
perturbed branch cut starting from $\pm\tilde{k}_{l}$ where $\tilde{k}%
_{l}=k_{l}+\epsilon_{l}\i $}%
\label{fig:picctrsmf}%
\end{figure}

Let $\lambda_{\nu}$ be a real pole of $\sigma_{ts}^{\ast\star}(\lambda) $ in
the integrand of \cref{eq-urf-decomp-int}, which is known as a surface wave
pole \cite{hu09, snyder12}. Integration across the surface wave pole is
understood as the limiting case of the perturbed system with lossy media as
mentioned above. For simplicity, suppose $\sigma_{ts}^{\ast\star}%
(\lambda)=\sigma(\lambda;k_{1},\cdots,k_{L})$ is the limit of the perturbed
field $\sigma(\lambda;\tilde{k}_{1},\cdots,\tilde{k}_{L})$ with pole
$\tilde{\lambda}_{\nu}$, and $\tilde{\lambda}_{\nu}\rightarrow\lambda_{\nu}%
\in(a,b)$ as all the $\epsilon_{l}\rightarrow0^{+}$. Let $\sigma_{\nu}%
=\lim_{\lambda\rightarrow\lambda_{\nu}}\sigma(\lambda)(\lambda-\lambda_{\nu}%
)$. Given any smooth function $h(\lambda)$, the limiting integral $\int
_{a}^{b}h(\lambda)\sigma(\lambda)d\lambda$ is evaluated by the formula
\begin{equation}
\begin{aligned} \int_a^b h(\lambda)\sigma(\lambda;\tilde{k}_1,\cdots,\tilde{k}_L)d\lambda \to & \int_a^b \left(h(\lambda)\sigma(\lambda)-\frac{h(\lambda_\nu)\sigma_\nu}{\lambda-\lambda_\nu}\right)d\lambda \\ & + \mathrm{p.v.}\int_{a}^{b} \frac{h(\lambda_\nu)\sigma_\nu}{\lambda-\lambda_\nu} d\lambda \pm \i\pi h(\lambda_\nu)\sigma_\nu, \end{aligned}\label{eq-int-pole}%
\end{equation}
here the $\pm$ sign is positive (or negative) when the perturbed pole
$\tilde{\lambda}_{\nu}\rightarrow\lambda_{\nu}$ from the upper (or the lower)
half of the complex plane, and the principal value part vanishes if
$(a,b)=(-\infty,+\infty)$.

In a well-posed physical problem, the poles will be at most of order one, and
$\tilde{\lambda}_{\nu}$ should keep in one side of the half planes as all the
perturbation parameters $\epsilon_{l}$ are sufficiently small, otherwise the
limit of the integral does not exist and the field is not well-defined. Also,
$0$ can not be a surface wave pole, otherwise the surface wave does not
propagate \cite{hu09, snyder12}.
\begin{remark}
Modes of the layered system are classified as the radiation modes, the guided modes (corresponding to the real poles) and the leaky modes (corresponding to other complex poles) \cite{hu09}.
\end{remark}

\subsection{The far-field expansions and their exponential convergence}

Here, we derive the far-field expansions for each integral $u_{ts}^{\ast\star
}\left(  \mathbf{x},\mathbf{x}^{\prime};\sigma_{ts}^{\ast\star}\right) $
in a natural generalization of the free-space case discussed in
\Cref{section-free-space}, then show their exponential convergence. The
derivation relies on the following two types of series expansions.

Suppose $(\rho_{0},\theta_{0})$ are the polar coordinates of $(x_{0},y_{0})$.
Denote
\begin{equation}
w_{l}(\lambda)=\frac{\lambda-\sqrt{\lambda^{2}-k_{l}^{2}}}{k_{l}},\text{
}0\leq l\leq L.
\end{equation}
By using the Bessel generating function \cref{eq-genBessel}, we have
\begin{align}%
\begin{split}
\label{eq-ME-deriv1}e^{-\sqrt{\lambda^{2}-k_{s}^{2}}\tau^{\star}y_{0}%
-\i\lambda x_{0}} ={}  &  g\left(  k_{s} \rho_{0}, -\i e^{\i\tau^{\star}%
\theta_{0}}w_{s}(\lambda) \right) \\
={}  &  \sum_{p=-\infty}^{\infty}J_{p}(k_{s}\rho_{0})e^{\i p \tau^{\star
}\theta_{0}} \cdot\left(  -\i w_{s}(\lambda) \right)  ^{p},
\end{split}
\\%
\begin{split}
\label{eq-ME-deriv2}e^{-\sqrt{\lambda^{2}-k_{t}^{2}}\tau^{\ast}y_{0}+\i\lambda
x_{0}} ={}  &  g\left(  k_{t} \rho_{0}, \i e^{\i\tau^{\ast}\theta_{0}}%
w_{t}(\lambda)^{-1} \right) \\
={}  &  \sum_{m=-\infty}^{\infty}J_{m}(k_{s}\rho_{0})e^{\i m \tau^{\ast}%
\theta_{0}} \cdot\left(  \i w_{t}(\lambda)^{-1} \right)  ^{m}.
\end{split}
\end{align}

For the ME, we split the difference $\mathbf{x}-\mathbf{x}^{\prime}=\left(
\mathbf{x}-\mathbf{x}_{c}\right)  +\left(  \mathbf{x}_{c}-\mathbf{x}^{\prime
}\right)  $, namely, we shift the source $\mathbf{x}^{\prime}$ to a common
source center $\mathbf{x}_{c}=(x_{c},y_{c})$ (assumed to be on the same side
of the interface $y=d_{s}^{\star}$, i.e. $y_{c}-d_{s}^{\star}$ and $y^{\prime
}-d_{s}^{\star}$ have the same sign). Let $(\rho_{c}^{\prime},\theta
_{c}^{\prime})$ be the polar coordinates of $\mathbf{x}^{\prime}%
-\mathbf{x}_{c}$. Using \cref{eq-ME-deriv1} with $(\rho_{0},\theta_{0}%
)=(\rho_{c}^{\prime},\theta_{c}^{\prime})$ and the separability of the plane
wave factor $\mathcal{E}_{ts}^{\star\star}(\mathbf{x},\mathbf{x}^{\prime})$
(\ref{eq-urf-decomp-int}), we get an approximation,
\begin{equation}
\begin{aligned} u_{ts}^{\ast\star}\left( \mathbf{x}, \mathbf{x}'; \sigma_{ts}^{\ast \star} \right) &=\int_{-\infty}^{\infty}\mathcal{E}_{ts}^{\ast\star}(\mathbf{x},\mathbf{x}',\lambda)\sigma_{ts}^{\ast \star}(\lambda)d\lambda \\ &=\int_{-\infty}^{\infty}\mathcal{E}_{ts}^{\ast\star}(\mathbf{x},\mathbf{x}_c,\lambda) \sigma_{ts}^{\ast \star}(\lambda) e^{-\sqrt{\lambda^2-k_s^2}\tau^\star(y'-y_c)+\i\lambda(x_c-x')} d\lambda \\ &= \int_{-\infty}^{\infty}\mathcal{E}_{ts}^{\ast\star}(\mathbf{x},\mathbf{x}_c,\lambda) \sigma_{ts}^{\ast \star}(\lambda)\sum_{p=-\infty}^{\infty}J_p(k_s\rho_c')e^{\i p \tau^\star \theta_c'} \left( -\i w_s(\lambda) \right)^p d\lambda\\ &\approx \sum_{|p|<P} I_{p}^{\ast\star}(\mathbf{x},\mathbf{x}_c)M_p^\star(\mathbf{x}',\mathbf{x}_c) \end{aligned}\label{eq-ME}%
\end{equation}
where the expansion function
\begin{equation}
I_{p}^{\ast\star}(\mathbf{x},\mathbf{x}_{c})=\int_{-\infty}^{\infty
}\mathcal{E}_{ts}^{\ast\star}(\mathbf{x},\mathbf{x}_{c},\lambda)\sigma
_{ts}^{\ast\star}(\lambda)\left(  -\i  w_{s}(\lambda)\right)  ^{p}d\lambda,
\end{equation}
and the ME cofficient
\begin{equation}
M_{p}^{\star}(\mathbf{x}^{\prime},\mathbf{x}_{c})=J_{p}(k_{s}\rho_{c}^{\prime
})e^{\i  p\tau^{\star}\theta_{c}^{\prime}}.
\end{equation}

For LE, we split the difference $\mathbf{x}-\mathbf{x}^{\prime}=\left(
\mathbf{x}-\mathbf{x}_{c}^{l}\right)  +\left(  \mathbf{x}_{c}^{l}%
-\mathbf{x}^{\prime}\right)  $, namely, we shift the target $\mathbf{x}$ to a
common target (local) center $\mathbf{x}_{c}^{l}=(x_{c}^{l},y_{c}^{l})$
(assumed to be on the same side of the interface $y=d_{t}^{\ast}$). Let
$(\rho^{l},\theta^{l})$ be the polar coordinates of $\mathbf{x}-\mathbf{x}%
_{c}^{l}$. Using \cref{eq-ME-deriv2} with $(\rho_{0},\theta_{0})=(\rho
^{l},\theta^{l})$ and the separability of the plane wave factor $\mathcal{E}%
_{ts}^{\star\star}(\mathbf{x},\mathbf{x}^{\prime})$ (\ref{eq-urf-decomp-int}),
we get an approximation,
\begin{equation}
\begin{aligned} &u_{ts}^{\ast\star}\left( \mathbf{x}, \mathbf{x}'; \sigma_{ts}^{\ast \star} \right) \\ =& \int_{-\infty}^{\infty}\mathcal{E}_{ts}^{\ast\star}(\mathbf{x}_c^l,\mathbf{x}',\lambda) \sigma_{ts}^{\ast \star}(\lambda) \sum_{m=-\infty}^{\infty}J_m(k_t\rho^l)e^{\i m \tau^\ast \theta^l} \cdot \left( \i w_t(\lambda)^{-1} \right)^m d\lambda\\ \approx & \sum_{|m|<M} L_{m}^{\ast\star}(\mathbf{x}_c^l,\mathbf{x}') K_m^{\ast}(\mathbf{x},\mathbf{x}_c^l) \end{aligned}\label{eq-LE}%
\end{equation}
where the expansion function
\begin{equation}
K_{m}^{\ast}(\mathbf{x},\mathbf{x}_{c}^{l})=J_{m}(k_{t}\rho^{l})e^{\i
m\tau^{\ast}\theta^{l}},
\end{equation}
and the LE coefficient
\begin{equation}
L_{m}^{\ast\star}(\mathbf{x}_{c}^{l},\mathbf{x}^{\prime})=\int_{-\infty
}^{\infty}\mathcal{E}_{ts}^{\ast\star}(\mathbf{x}_{c}^{l},\mathbf{x}^{\prime
},\lambda)\sigma_{ts}^{\ast\star}(\lambda)\left(  \i  w_{t}(\lambda
)^{-1}\right)  ^{m}d\lambda.
\end{equation}

The M2L can be derived directly by using the
splitting $\mathbf{x}_{c}^{l}-\mathbf{x}^{\prime}=\left(  \mathbf{x}_{c}%
^{l}-\mathbf{x}_{c}\right)  +\left(  \mathbf{x}_{c}-\mathbf{x}^{\prime
}\right)  $ in $L_{m}^{\ast\star}(\mathbf{x}_{c}^{l},\mathbf{x}^{\prime})$,%
\begin{equation}
\begin{aligned} &L_{m}^{\ast\star}(\mathbf{x}_c^l,\mathbf{x}') \\ =& \int_{-\infty}^{\infty} \mathcal{E}_{st}^{\ast\star}(\mathbf{x}_c^l,\mathbf{x}_c,\lambda) \sigma_{st}^{\ast\star}(\lambda)\left(\i w_t(\lambda)^{-1}\right)^{m} \sum_{p=-\infty}^{\infty}J_p(k_s\rho_c')e^{\i p \tau^\star \theta_c'}\cdot\left(-\i w_s(\lambda) \right)^p d\lambda \\ \approx & \sum_{|p|<P} A_{mp}^{\ast\star}(\mathbf{x}_c^l,\mathbf{x}_c)M_p^{\star}(\mathbf{x}',\mathbf{x}_c). \end{aligned}\label{eq-M2L}%
\end{equation}
where the translation coefficients $A_{mp}^{\ast\star}(\mathbf{x}_{c}%
^{l},\mathbf{x}_{c})$ are given by
\[
A_{mp}^{\ast\star}(\mathbf{x}_{c}^{l},\mathbf{x}_{c})=\int_{-\infty}^{\infty
}\mathcal{E}_{ts}^{\ast\star}(\mathbf{x}_{c}^{l},\mathbf{x}_{c},\lambda
)\sigma_{ts}^{\ast\star}(\lambda)\left(  -\i  w_{s}(\lambda)\right)
^{p}\left(  \i  w_{t}(\lambda)^{-1}\right)  ^{m}d\lambda.
\]

The L2L shifts the local center $\mathbf{x}_{c}^{l}$ in
each integral $L_{m}^{\ast\star}(\mathbf{x}_{c}^{l},\mathbf{x}^{\prime})$ to a
new local center $\tilde{\mathbf{x}}_{c}^{l} = (\tilde{x}_{c}^{l},\tilde
{y}_{c}^{l})$. Let $(\tilde{\rho},\tilde{\theta})$ be the polar coordinates of
$\tilde{\mathbf{x}}_{c}^{l} - \mathbf{x}_{c}^{l}$. Using \cref{eq-ME-deriv2}
with $(\rho_{0},\theta_{0}) = (\tilde{\rho},\tilde{\theta})$,
\begin{equation}
\label{eq-L2L}\begin{aligned} &L_{m}^{\ast\star}(\tilde{\bd{x}}_c^l,\mathbf{x}')\\ =&\int_{-\infty}^{\infty}\mathcal{E}_{ts}^{\ast\star}(\mathbf{x}_c^l,\mathbf{x}',\lambda) \sigma_{ts}^{\ast \star}(\lambda)\left( \i w_t(\lambda)^{-1} \right)^m \sum_{p=-\infty}^{\infty} J_p(k_t \tilde{\rho})e^{\i p \tau^\ast \tilde{\theta}}\cdot \left( \i w_t(\lambda)^{-1} \right)^p d\lambda \\ \approx & \sum_{|p+m|<P}J_p(k_t \tilde{\rho})e^{\i p \tau^\ast \tilde{\theta}} \int_{-\infty}^{\infty}\mathcal{E}_{ts}^{\ast\star}(\mathbf{x}_c^l,\mathbf{x}',\lambda) \sigma_{ts}^{\ast \star}(\lambda)\left( \i w_t(\lambda)^{-1} \right)^m \left( \i w_t(\lambda)^{-1} \right)^p d\lambda \\ =& \sum_{|p|<P} L_{p}^{\ast\star}(\mathbf{x}_c^l,\mathbf{x}')K_{p-m}^{\ast}(\tilde{\bd{x}}_c^l,\mathbf{x}_c^l). \end{aligned}
\end{equation}

Next, before we present the main result of this paper on the convergence of
the series expansions above, we introduce the concept of \textquotedblleft
polarized distance\textquotedblright\ unique to the interaction in layered
media. Given layer indices $s$, $t$ and direction marks $\ast,\star
\in\{\uparrow,\downarrow\}$, for points $\mathbf{x}_{1}=(x_{1},y_{1})$ and
$\mathbf{x}_{2}=(x_{2},y_{2})$, a \textquotedblleft polarized
distance\textquotedblright\ is defined as
\begin{equation}
D_{ts}^{\ast\star}(\mathbf{x}_{1},\mathbf{x}_{2})=\sqrt{(x_{1}-x_{2}%
)^{2}+\left(  \tau^{\ast}(y_{1}-d_{t}^{\ast})+\tau^{\star}(y_{2}-d_{s}^{\star
})\right)  ^{2}},
\label{poldist}
\end{equation}
provided both $\tau^{\ast}(y_{1}-d_{t}^{\ast})>0$ and $\tau^{\star}%
(y_{2}-d_{s}^{\star})>0$. (Note that they are not symmetric with respect to
$\mathbf{x}_{1}$ and $\mathbf{x}_{2}$.)

\begin{theorem}[exponential convergence of far-field expansions in layered media]\label{thm-exp-conv}Suppose the integral $u_{ts}^{\ast\star}\left( \bd{x}, \bd{x}'; \sigma_{ts}^{\ast \star} \right)$ is derived from a well-posed Helmholtz problem in layered media as in \cref{lemma-decomp-rf}.
Then, we have the truncation error of ME \cref{eq-ME}
\begin{equation}\label{eq-err-ME}
\left| u_{ts}^{\ast\star}\left( \bd{x}, \bd{x}'; \sigma_{ts}^{\ast \star} \right) -\sum_{|p|<P} I_{p}^{\ast\star}(\bd{x},\bd{x}_c)M_{p}^\star(\bd{x}',\bd{x}_c) \right| \le c^{\mathrm{ME}}(P) \left( \frac{|\bd{x}'-\bd{x}_c|}{D_{ts}^{\ast\star}(\bd{x},\bd{x}_c)} \right)^P,
\end{equation}
the truncation error of LE \cref{eq-LE}
\begin{equation}\label{eq-err-LE}
\left|u_{ts}^{\ast\star}\left( \bd{x}, \bd{x}'; \sigma_{ts}^{\ast \star} \right) - \sum_{|m|<M} L_{m}^{\ast\star}(\bd{x}_c^l,\bd{x}')K_m^{\ast}(\bd{x},\bd{x}_c^l) \right| \le c^{\mathrm{LE}}(M) \left( \frac{|\bd{x}-\bd{x}_c^l|}{D_{ts}^{\ast\star}(\bd{x}_c^l,\bd{x}')} \right)^M,
\end{equation}
the truncation error of M2L \cref{eq-M2L} for each LE coefficient
\begin{equation}\label{eq-err-M2L}
\left|L_{m}^{\ast\star}(\bd{x}_c^l,\bd{x}') - \sum_{|p|<P} A_{mp}^{\ast\star}(\bd{x}_c^l,\bd{x}_c)M_{p}^\star(\bd{x}',\bd{x}_c)\right| \le c^{\mathrm{M2L}}_m(P) \left( \frac{|\bd{x}'-\bd{x}_c|}{D_{ts}^{\ast\star}(\bd{x}_c^l,\bd{x}_c)} \right)^P
\end{equation}
and the truncation error of L2L \cref{eq-L2L} for each LE coefficient
\begin{equation}\label{eq-err-L2L}
\left|L_{m}^{\ast\star}(\tilde{\bd{x}}_c^l,\bd{x}')-\sum_{|p|<P} L_{p}^{\ast\star}(\bd{x}_c^l,\bd{x}')K_{p-m}^{\ast}(\tilde{\bd{x}}_c^l,\bd{x}_c^l)\right| \le c^{\mathrm{L2L}}_m(P)\left( \frac{|\tilde{\bd{x}}_{c}^{l} - \bd{x}_{c}^{l}|}{D_{ts}^{\ast\star}(\bd{x}_c^l,\bd{x}')} \right)^P
\end{equation}
for some functions $c^{\mathrm{ME}}(\cdot)$, $c^{\mathrm{LE}}(\cdot)$, $c_{m}^{\mathrm{M2L}}(\cdot)$ and $c_{m}^{\mathrm{L2L}}(\cdot)$ having polynomial growth rates, provided that for some given $c_0>1$, the far-field conditions
\begin{align}\label{eq-far-field-condition}
\begin{split}
D_{ts}^{\ast\star}(\bd{x},\bd{x}_c) > c_0 |\bd{x}'-\bd{x}_c|,&\quad D_{ts}^{\ast\star}(\bd{x}_c^l,\bd{x}') > c_0 |\bd{x}-\bd{x}_c^l|, \\
\text{ } D_{ts}^{\ast\star}(\bd{x}_c^l,\bd{x}_c) > c_0 |\bd{x}'-\bd{x}_c|,&\quad D_{ts}^{\ast\star}(\bd{x}_c^l,\bd{x}') > c_0 |\tilde{\bd{x}}_{c}^{l} - \bd{x}_{c}^{l}|
\end{split}
\end{align}
hold, respectively.
If all the sources, targets and the centers involved above are bounded by a given box, the distances from every center to its nearby interface have a given nonzero lower bound, and there exist $0<\rho_m \le \rho_M$ such that
\begin{equation*}
\rho_m \le D_{ts}^{\ast\star}(\bd{x},\bd{x}_c),D_{ts}^{\ast\star}(\bd{x}_c^l,\bd{x}'),D_{ts}^{\ast\star}(\bd{x}_c^l,\bd{x}_c),D_{ts}^{\ast\star}(\bd{x}_c^l,\bd{x}')\le \rho_M,
\end{equation*}
then the functions $c^{\mathrm{ME}}(\cdot)$, $c^{\mathrm{LE}}(\cdot)$, $c_{m}^{\mathrm{M2L}}(\cdot)$ and $c_{m}^{\mathrm{L2L}}(\cdot)$ can be chosen to be determined by these bounds, without dependence on the actual positions of the source locations.
\end{theorem}

The proof will be a special case of a more general convergence result of the
Bessel-type expansions in \cref{thm-Bessel-type-expansion}, to be given in \Cref{sect-3}.

\subsection{Implementation of a FMM framework for sources in layered media}

In the far-field conditions \cref{eq-far-field-condition} of the convergence
results, the polarized distances $D_{ts}^{\ast\star}$ play the role of the
far-field distances as in the free-space cases, which will affect how the ME
based FMM\ will be implemented. \begin{figure}[tbh]
\captionsetup[subfigure]{labelformat=empty} \centering \subfloat
{\includegraphics[scale=0.3]{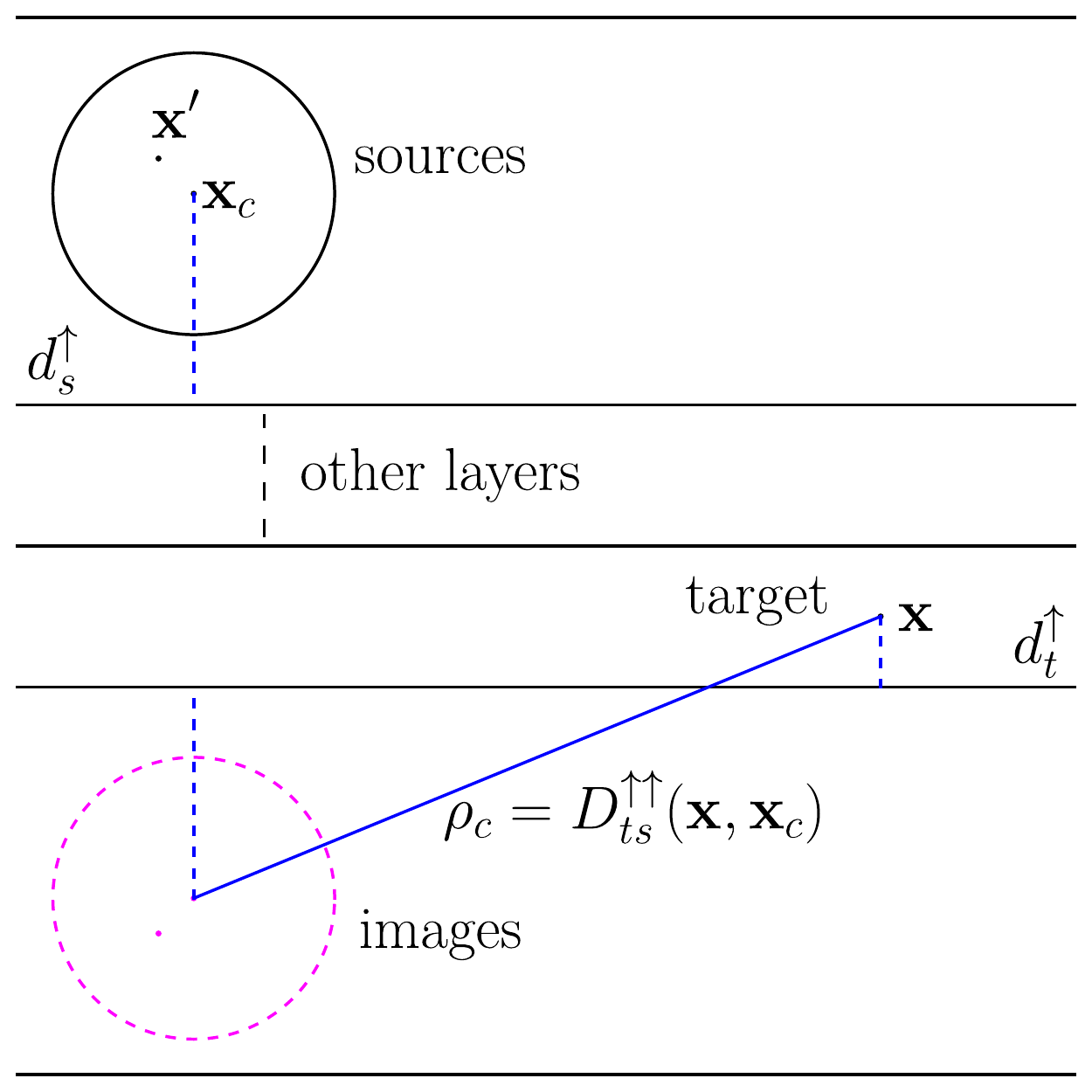}} \hspace{0.06\textwidth
}\subfloat
{\includegraphics[scale=0.3]{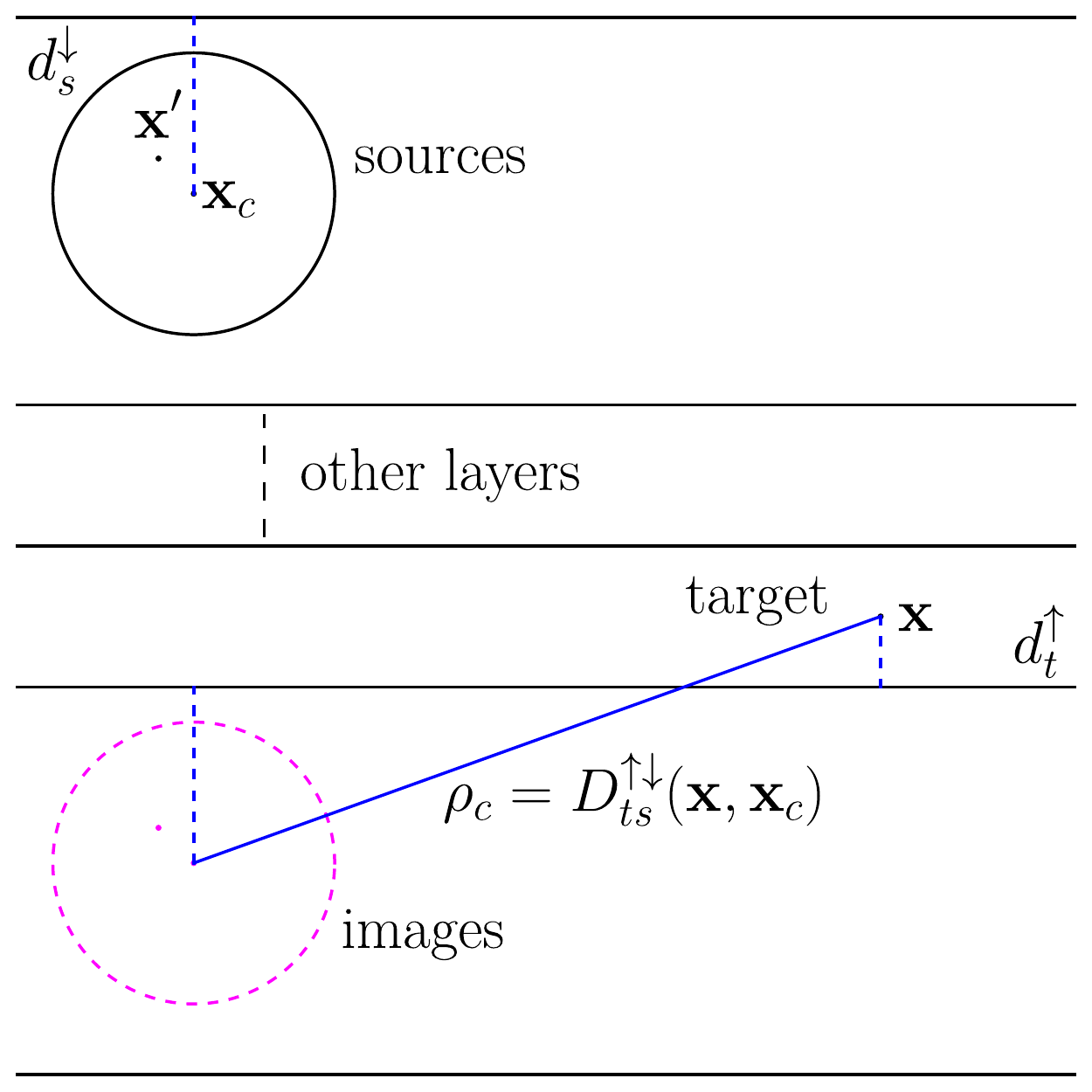}}
\vskip\baselineskip
\subfloat
{\includegraphics[scale=0.3]{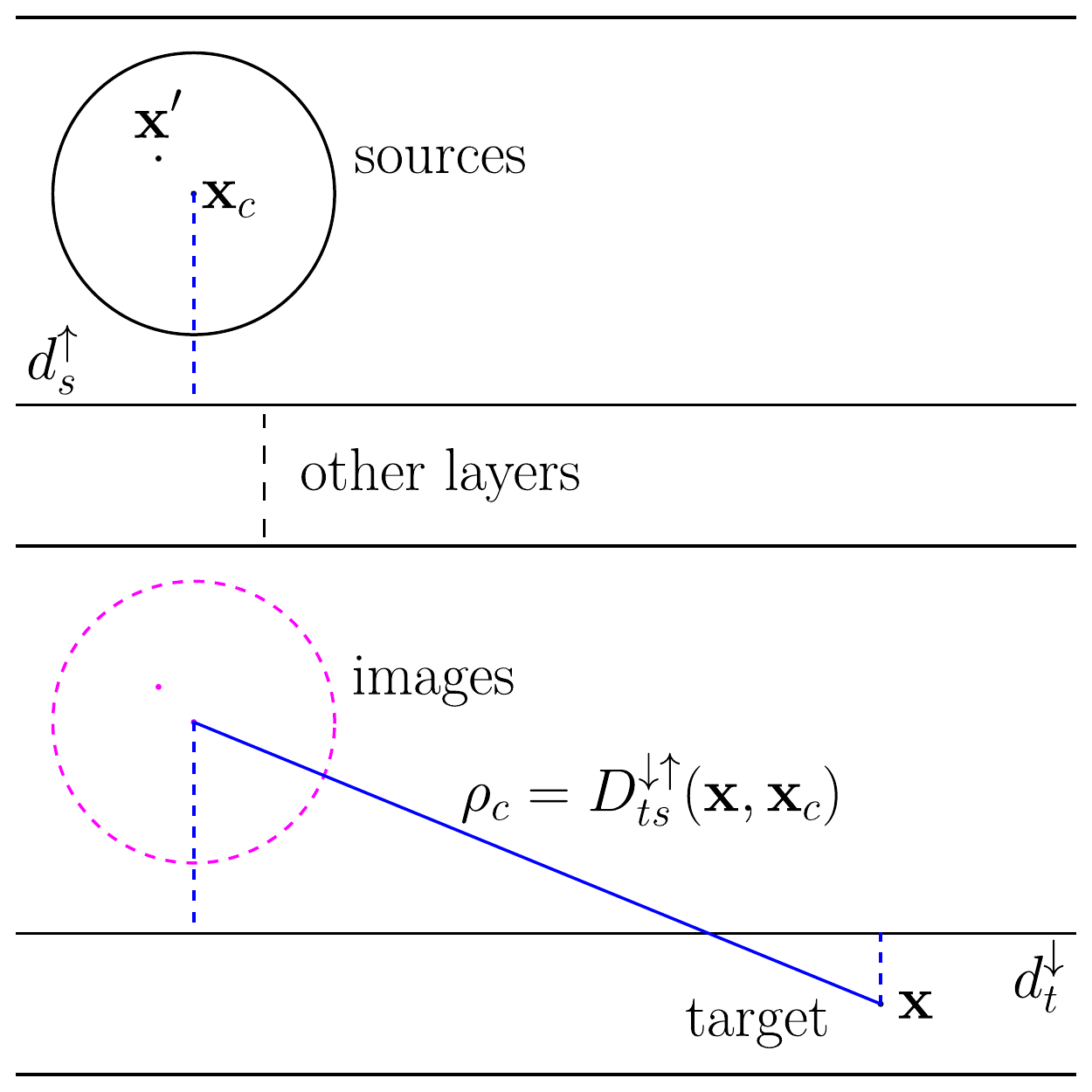}} \hspace{0.06\textwidth
}\subfloat
{\includegraphics[scale=0.3]{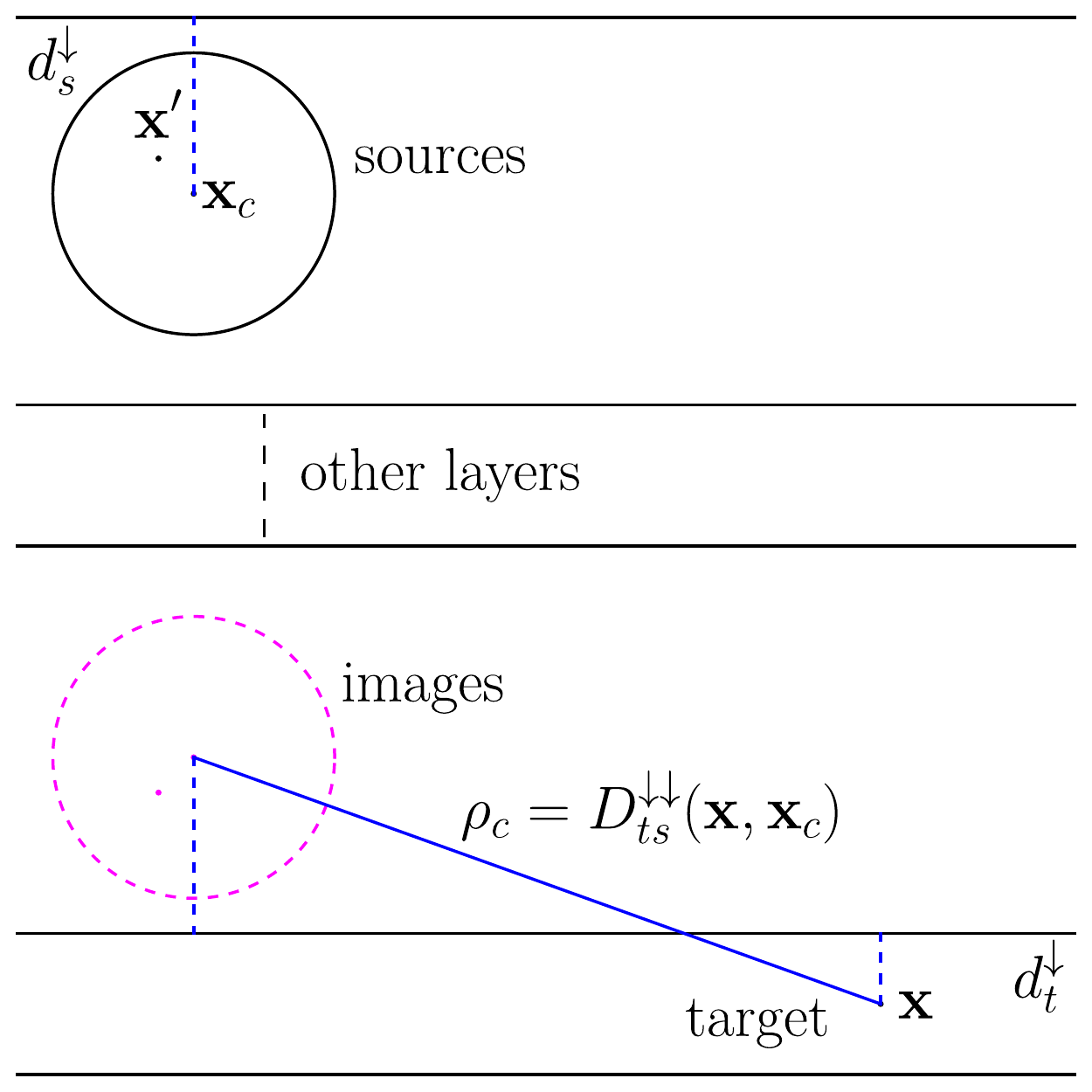}}\caption{The far-field
distance $\rho_{c}$ of the ME in various field propagation directions.}%
\label{fig:picme}%
\end{figure}

Define a bijective linear mapping
\begin{equation}
\mathcal{P}_{ts}^{\ast\star}:\mathbf{x}_{2}=(x_{2},y_{2})\mapsto
\tilde{\mathbf{x}}_{2}=\left(  x_{2},d_{t}^{\ast}-\tau^{\ast}\tau^{\star
}\left(  y_{2}-d_{s}^{\star}\right)  \right)  \label{map-pol-source}%
\end{equation}
provided $\tau^{\star}\left(  y_{2}-d_{s}^{\star}\right)  >0$. It is
straightforward that
\begin{equation}
D_{ts}^{\ast\star}(\mathbf{x}_{1},\mathbf{x}_{2})=\left\Vert \mathbf{x}%
_{1}-\mathcal{P}_{ts}^{\ast\star}(\mathbf{x}_{2})\right\Vert
,\label{eq-Euclidean}%
\end{equation}
here $\Vert\cdot\Vert$ is the Euclidean norm. \Cref{fig:picme} shows how
$\mathcal{P}_{ts}^{\ast\star}$ maps the sources to their \textquotedblleft
polarization images\textquotedblright\ and the far-field distance of the ME
should be $D_{ts}^{\ast\star}(\mathbf{x},\mathbf{x}_{c})$ for various reaction
component of the Green's function.

The FMM for layered media can be set up to evaluate each reaction
component $u_{ts}^{\ast\star}$ as follows: $\mathcal{P}_{ts}^{\ast\star}$ maps
the source layer $s$ to a neighboring layer (below or above) of the target
layer $t$, where all the far-field distances become Euclidean as in
\cref{eq-Euclidean}. Therefore, to
calculate the interaction due to any of the reaction component $u_{ts}%
^{\ast\star}\left(  \mathbf{x},\mathbf{x}^{\prime};\sigma_{ts}^{\ast\star
}\right)  $, we simply move the source charges to the locations of their
corresponding \textquotedblleft polarization images\textquotedblright. An
implementation for Helmholtz equations in 3-D layered media based on this
approach is given in \cite{bo2019sphfmm}.

\section{The convergence estimate on Bessel-type expansions}

\label{sect-3} In this section, we will give convergence estimates on general
Bessel-type expansions, of which \cref{thm-exp-conv} will be a special case.

The Bessel-type expansions are defined as follows. Let $k>0$, $(\rho,\theta)$,
$(\rho^{\prime},\theta^{\prime})$ be the polar coordinates of $\mathbf{x}%
=(x,y)$ and $\mathbf{x}^{\prime}=(x^{\prime},y^{\prime})$, respectively.
Suppose $y>0$, $y+y^{\prime}>0$ and $\rho>\rho^{\prime}\geq0$. For simplicity,
define
\begin{equation}
\Psi(\lambda)\equiv\Psi(\mathbf{x,}\lambda)=e^{-\sqrt{\lambda^{2}-k^{2}}%
y+\i\lambda x},\quad\Psi^{\prime}(\lambda)\equiv\Psi^{\prime}(\mathbf{x}%
^{\prime}\mathbf{,}\lambda)=e^{-\sqrt{\lambda^{2}-k^{2}}y^{\prime}-\i\lambda
x^{\prime}}.
\end{equation}
Then, we claim the \emph{pointwise} \emph{Bessel-type expansion} for a
given $\lambda_{\nu}\in\mathbb{C}$,
\begin{equation}
e^{-\sqrt{\lambda_{\nu}^{2}-k^{2}}(y+y^{\prime})+\i\lambda_{\nu}(x-x^{\prime
})}=\sum_{p=-\infty}^{\infty}J_{p}(k\rho^{\prime})e^{\i p\theta^{\prime}}%
\Psi(\mathbf{x,}\lambda_{\nu})\left(  -\i w(\lambda_{\nu})\right)
^{p}\label{eq-Bessel-type-expansion-discrete}%
\end{equation}
and the \emph{integral} \emph{Bessel-type expansion} for the integration over
$\lambda\in\lbrack a,b],-\infty\leq a<b\leq+\infty,$
\begin{equation}
\int_{a}^{b}e^{-\sqrt{\lambda^{2}-k^{2}}(y+y^{\prime})+\i\lambda(x-x^{\prime
})}f(\lambda)d\lambda=\sum_{p=-\infty}^{\infty}J_{p}(k\rho^{\prime})e^{\i
p\theta^{\prime}}F_{p}(x,y),\label{eq-Bessel-type-expansion}%
\end{equation}
where $f(\lambda)$ is a complex function defined on $(a,b)$ satisfying certain
conditions to be specified later, and $F_{p}(x,y)$ is the expansion function
\[
F_{p}(x,y)=\int_{a}^{b}\Psi(\mathbf{x,}\lambda)\left(  -\i w(\lambda)\right)
^{p}f(\lambda)d\lambda.
\]

\subsection{Convergence of pointwise Bessel-type expansions}

We first present the convergence of \cref{eq-Bessel-type-expansion-discrete}.
\begin{lemma}\label{lemma-pole}Let $c_0 > 1$, $k>0$.
Suppose $(\rho',\theta')$ are the polar coordinates of $(x',y')$, $x\in\mathbb{R}$, $y\in\mathbb{R}^+$ satisfying $\rho = \sqrt{x^2+y^2} > c_0 \rho' \ge 0$ and $x\cdot\Im \lambda_\nu \ge 0$.
Then, the Bessel-type expansion \cref{eq-Bessel-type-expansion-discrete} holds with a truncation error estimate
\begin{equation}
\left|\sum_{|p| \ge P} J_{p}(k\rho') e^{\i p \theta'} \Psi(\mathbf{x,}\lambda_{\nu}) \left(-\i w(\lambda_\nu)\right)^p \right| \le \frac{2c_0}{c_0 -1}\left( \frac{\rho'}{\rho} \right)^P
\end{equation}
for any $P \ge e(|\lambda_\nu|+k/2)\rho$.
\end{lemma}\begin{proof}
The equality of \cref{eq-Bessel-type-expansion-discrete} is given by the Bessel generating function \cref{eq-genBessel}
\begin{align}\label{eq-disc-Bessel-expansion}
\begin{split}
e^{-\sqrt{\lambda_\nu^{2}-k^{2}}(y+y')+\i\lambda_\nu(x-x')}
&= \Psi(\lambda_\nu) g\left( k\rho', -\i e^{\i \theta'} w(\lambda_\nu) \right) \\
&=\text{ }\sum_{p=-\infty}^{\infty} J_{p}(k\rho') e^{\i p \theta'} \mathcal{E}(\lambda_\nu) \left(-\i w(\lambda_\nu)\right)^p.
\end{split}
\end{align}
With the given conditions, $|\exp({-\sqrt{\lambda_\nu^2-k^2}y+\i \lambda_\nu x}) | \le 1$, $|w(\lambda_\nu)| \le (2|\lambda_\nu|+k)/{k}$, hence for each $p$, using \cref{lemma-Jn-bound},
\begin{equation*}
\left|J_{p}(k\rho') e^{\i p \theta'} e^{-\sqrt{\lambda_\nu^{2}-k^{2}}y+\i\lambda_\nu x} \left(-\i w(\lambda_\nu)\right)^p \right| \le \frac{1}{|p|!}\left(\frac{k\rho'}{2}\right)^{|p|} \left(\frac{2|\lambda_\nu|+k}{k}\right)^{|p|}.
\end{equation*}
For $|p| \ge e(|\lambda_\nu|+k/2)\rho$, using Stirling's formula \cite{Stirling1955},
\begin{equation*}
|p|! \ge \left(\frac{|p|}{e}\right)^{|p|} \ge \left( \left( |\lambda_\nu| + \frac{k}{2} \right) \rho \right) ^{|p|},
\end{equation*}
thus,
\begin{equation*}
\left|J_{p}(k\rho') e^{\i p \theta'} e^{-\sqrt{\lambda_\nu^{2}-k^{2}}y+\i\lambda_\nu x} \left(-\i w(\lambda_\nu)\right)^p \right| \le \left(\frac{\rho'}{\rho}\right)^{|p|},
\end{equation*}
which will give the estimate of the truncation error after summing over $|p|\ge P$.
\end{proof}

\subsection{Special cases of integral Bessel-type expansion}

First, we consider \cref{eq-Bessel-type-expansion} when the integral is
defined on a bounded interval $[-k^{\prime},k^{\prime}]$.
\begin{lemma}\label{lemma-bounded-interval}
Let $c_0>1$, $k' \ge k > 0$.
Let $(\rho,\theta)$ and $(\rho',\theta')$ be the polar coordinates of $\bd{x}=(x,y)$ and $\bd{x}'=(x',y')$, respectively.
Suppose $y>0$, $\rho > c_0\rho' \ge 0$, and the function ${f}(\lambda)$ on $[-k',k']$ satisfies $\int_{-k'}^{k'} |{f}(\lambda)|d\lambda = S < +\infty$, then the integral Bessel-type expansion \cref{eq-Bessel-type-expansion} holds on $[-k',k']$ with a truncation error estimate
\begin{equation}
\left|\sum_{|p|\ge P} J_{p}(k\rho') e^{\i p \theta'} F_{p}\right| \le \frac{2c_0 S}{c_0-1}\left( \frac{\rho'}{\rho} \right)^P
\end{equation}
for any $P \ge ek'\rho$.
\end{lemma}
\begin{proof}
When $|\lambda| \le k'$,
\begin{equation*}
\left|e^{-\sqrt{\lambda^{2}-k^{2}}y+\i\lambda x}\right| \le 1,\quad\left|w(\lambda)\right|^{\pm 1}=\left|  \frac{\lambda-\sqrt{\lambda^{2}-k^{2}}}{k}\right|^{\pm 1} \le \frac{2k'}{k},
\end{equation*}
so each $F_p$ is bounded by
\begin{equation*}
|F_p|\le \tilde{F}_p:= \int_{-k'}^{k'} \left| e^{-\sqrt{\lambda^{2}-k^{2}}y+\i\lambda x} \left(-\i w(\lambda)\right)^{p}f(\lambda)\right|d\lambda \le S \left( \frac{2k'}{k} \right)^{|p|}.
\end{equation*}
When $|p| \ge ek'\rho$, using Stirling's formula \cite{Stirling1955}, $|p|! \ge \left({|p|}/{e}\right)^{|p|} \ge (k'\rho)^{|p|}$, so
\begin{equation*}
\tilde{F}_p \le S\left( \frac{2k'}{k} \right)^{|p|} \le S |p|!\left(\frac{k\rho}{2}\right)^{-|p|}.
\end{equation*}
Hence for $|p|\ge ek'\rho$, by \cref{lemma-Jn-bound},
\begin{equation*}
\int_{-k'}^{k'}\left|J_{p}(k\rho') e^{\i p \theta'} e^{-\sqrt{\lambda^{2}-k^{2}}y+\i\lambda x}\left(-\i w(\lambda)\right)^{p}f(\lambda)\right|d\lambda \le\frac{1}{|p|!}\left( \frac{k\rho'}{2} \right)^{|p|}\tilde{F}_p\le S\left( \frac{\rho'}{\rho} \right)^{|p|}.
\end{equation*}
Similar as in \cref{eq-disc-Bessel-expansion} and using the Fubini's theorem,
\begin{equation*}
\begin{aligned}
\int_{-k'}^{k'} e^{-\sqrt{\lambda^{2}-k^{2}}(y+y')+\i
\lambda(x-x')} f(\lambda)d\lambda &= \int_{-k'}^{k'}{\sum_{p=-\infty}^{\infty}J_{p}(k\rho') e^{\i p \theta'} \Psi(\lambda)\left(-\i w(\lambda)\right)^{p}f(\lambda) d\lambda} \\
&= \sum_{p=-\infty}^{\infty} J_{p}(k\rho') e^{\i p \theta'} F_{p}.
\end{aligned}
\end{equation*}
For $P \ge ek'\rho$, the truncation error
\begin{equation*}
\left| \sum_{|p| \ge P} J_{p}(k\rho') e^{\i p \theta'} F_{p} \right| \le \sum_{|p|\ge P}S\left( \frac{\rho'}{\rho} \right)^{|p|} \le \frac{2c_0 S}{c_0-1}\left( \frac{\rho'}{\rho} \right)^P.
\end{equation*}
\end{proof}

A similar result can be derived for a complex path of finite length.
\begin{lemma}\label{lemma-bounded-contour}
Let $c_0 > 1$, $k>0$.
Let $(\rho,\theta)$ and $(\rho',\theta')$ be the polar coordinates of $\bd{x}=(x,y)$ and $\bd{x}'=(x',y')$, respectively.
Suppose $y>0$, $\rho > c_0 \rho' \ge 0$.
Let $\kappa \subset \mathbb{C}$ be a smooth curve with length $|\kappa|<+\infty$.
Suppose $x\cdot \Im \lambda \ge 0$ for any $\lambda \in \kappa$.
Let $f(\lambda)$ be a complex function defined on $\kappa$ satisfying $|f(\lambda)| \le f_0$.
Then,
\begin{equation}\label{eq-kappa-exp}
E_{\kappa}=\int_{\kappa} \Psi(\lambda) \Psi'(\lambda)f(\lambda)d\lambda=
\sum_{p=-\infty}^{\infty} J_{p}(k\rho') e^{\i p \theta'} \int_{\kappa}\Psi(\lambda) \left(-\i w(\lambda)\right)^p f(\lambda) d\lambda,
\end{equation}
with a truncation error estimate
\begin{equation}
\left|\sum_{|p|\ge P} J_{p}(k\rho') e^{\i p \theta'} \int_{\kappa}\Psi(\lambda) \left(-\i w(\lambda)\right)^p f(\lambda) d\lambda \right| \le \frac{2c_0 f_0 |\kappa|}{c_0-1}\left( \frac{\rho'}{\rho} \right)^P
\end{equation}
for any $P \ge e(\lambda_M + k/2)\rho$, where $\lambda_M = \max_{\lambda\in\kappa}|\lambda|$.
\end{lemma}

\begin{proof}
Suppose $\kappa$ is parameterized by $\lambda = a(s)+b(s)\i$, $s \in [0,1]$, here $a(s)$ and $b(s)$ are real and smooth functions.
Using the results from the proof of \cref{lemma-pole} and \cref{lemma-bounded-interval}, for $\lambda \in \kappa$,
\begin{equation*}
\left|e^{-\sqrt{\lambda^2-k^2}y+\i\lambda x}\right| \le 1\text{, }|w(\lambda)| \le \frac{2\lambda_M+k}{k},
\end{equation*}
so for each $p$, using \cref{lemma-Jn-bound},
\begin{align*}
&\int_{0}^{1} \left| J_p(k\rho')e^{\i p \theta'} \Psi(\lambda)\left( -\i w(\lambda) \right)^p f(\lambda) \left(a'(s)+b'(s)\i\right) \right| ds \\
\le & \frac{1}{|p|!}\left(\frac{k\rho'}{2}\right)^p \cdot 1 \cdot \left(\frac{2\lambda_M+k}{k}\right)^p \cdot f_0 |\kappa|.
\end{align*}
Hence, using the Bessel generating function \cref{eq-genBessel} and the Fubini's theorem,
\begin{align*}
E_{\kappa}=& \int_{0}^{1} \Psi(\lambda)\Psi'(\lambda) f(\lambda) \left(a'(s)+b'(s)\i\right) ds \\
=&\sum_{p=-\infty}^{\infty}\int_{0}^{1} J_p(k\rho')e^{\i p \theta'} \Psi(\lambda)\left( -\i w(\lambda) \right)^p f(\lambda) \left(a'(s)+b'(s)\i\right) ds \\
=& \sum_{p=-\infty}^{\infty}\int_{\kappa} J_p(k\rho')e^{\i p \theta'} \Psi(\lambda)\left( -\i w(\lambda) \right)^p f(\lambda) d\lambda,
\end{align*}
we get the equality of \cref{eq-kappa-exp}.
The truncation error estimate is similar as in the proof of \cref{lemma-bounded-interval}, except from replacing $k'$ by $\lambda_M + k/2$.
\end{proof}

Next, we consider a special case when $(x,y)=(0,\rho)$ in the Bessel-type
expansion \cref{eq-Bessel-type-expansion} over an infinite interval.
\begin{lemma}\label{lemma-Epp}
Let $c_0>1$, $k' \ge k>0$, $x',y' \in \mathbb{R}$, $\rho >c_0 \rho' = \sqrt{x'^2+y'^2} \ge 0$. Let $C\in\mathbb{R}^+$ and $K$ be a nonnegative integer, and
\begin{align}
E_p^+ = \int_{k'}^{\infty}e^{-\sqrt{\lambda^2-k^2}\rho}\left(-\i w(\lambda)\right)^p f(\lambda) d\lambda,\text{ } p \in\mathbb{Z},
\end{align}
here $f(\lambda)$ satisfies $|f(\lambda)| \le C\lambda^{K}$ for $\lambda \in[k',\infty)$.
Then for any sufficiently large $|p|$ such that $|p| \ge (k\rho)^2/4 +1 -K$ we have the estimate
\begin{equation}
\left|E_p^{+}\right| \le \int_{k'}^{\infty}e^{-\sqrt{\lambda^2-k^2}\rho} w(\lambda)^p |f(\lambda)| d\lambda \le 3C\left(|p|+K\right)! \left(\frac{2}{\rho} \right)^{K+1} \left(\frac{k\rho}{2}\right)^{-|p|}.
\end{equation}
In addition, the Bessel-type expansion \cref{eq-Bessel-type-expansion} holds with $(x,y)=(0,\rho)$ on the interval $(k',\infty)$, with a truncation error estimate
\begin{equation}
\left|\int_{k'}^{\infty} e^{-\sqrt{\lambda^{2}-k^{2}}(\rho+y')+\i
\lambda(-x')} f(\lambda)d\lambda - \sum_{|p| < P} J_{p}(k\rho') e^{\i p \theta'} E_{p}^+ \right| \le c(P) \left(\frac{\rho'}{\rho} \right)^P
\end{equation}
for any $P \ge (k\rho)^2/4 +1-K$, where
\begin{equation}
c(P)=6C(K+1)!\left( \frac{2c_0}{\rho (c_0-1)} \right)^{K+1}(P+K)^K.
\end{equation}
\end{lemma}\begin{proof}\label{proof-lemma-Epp}
Notice that for $\lambda \ge k$ we have $\sqrt{\lambda^2-k^2} \le \lambda$ and $0 \le \lambda - \sqrt{\lambda^2-k^2} \le k \le \lambda \le \lambda + \sqrt{\lambda^2-k^2}$, so each $|E_p^+| \le C k^{K+1} I_p$, where
\begin{equation}
I_p=\int_{k}^{\infty} \frac{e^{-\sqrt{\lambda^2-k^2}\rho}}{\sqrt{\lambda^2-k^2}} \left(\frac{\lambda+\sqrt{\lambda^2-k^2}}{k}\right)^{M+1}d\lambda,
\end{equation}
here $M=|p|+K$.
With the substitution $v=(\lambda+\sqrt{\lambda^2-k^2})/k$,
\begin{align*}
I_p &= \int_{1}^{\infty} e^{\frac{k\rho}{2}(-v+v^{-1})}v^M dv \\
&\le \int_{1}^{\infty} e^{\frac{k\rho}{2}(-v)} \left( \sum_{j=0}^{M-1}\frac{1}{j!}\left(\frac{k\rho}{2}v^{-1}\right)^j + \frac{1}{M!}\left(\frac{k\rho}{2}v^{-1}\right)^M e^{\frac{k\rho}{2}v^{-1}} \right) v^M dv \\
&\le \sum_{j=0}^{M-1}\frac{1}{j!}\left(\frac{k\rho}{2}\right)^j \int_0^{\infty} e^{\frac{k\rho}{2}(-v)}v^{M-j}dv
+ \frac{1}{M!}\left(\frac{k\rho}{2}\right)^M\int_1^{\infty} e^{\frac{k\rho}{2}(-v+1)} dv \\
&= \sum_{j=0}^{M-1}\frac{(M-j)!}{j!}\left(\frac{k\rho}{2}\right)^{2j-M-1}+\frac{1}{M!}\left(\frac{k\rho}{2}\right)^{M-1} \\
&= M! \left(\frac{k\rho}{2} \right)^{-M-1} \sum_{j=0}^M c_j
\end{align*}
where
\begin{equation}
c_j = \frac{(M-j)!}{M!j!}\left(\frac{k\rho}{2}\right)^{2j}, \text{ } j = 0,\cdots,M.
\end{equation}
One can quickly verify $c_0 = 1$, $c_1 = \left({k\rho}/{2}\right)^{2} / M \le {(M-1)}/{M}$.
For $1 \le j \le M-2$, we have $c_{j+1} / c_j  = (k\rho)^2 / 4(j+1)(M-j) \le 1 / 2$.
For $c_M$ we have $c_M / c_{M-1} = (k\rho)^2/4M \le 1 $.
In sum, $\sum_{j=0}^{M} c_j \le c_0 + 2c_1 \le 1+2(M-1)/M \le 3$, so
\begin{align*}
|E_p^+| \le Ck^{K+1} I_p \le 3Ck^{K+1} M!\left(\frac{k\rho}{2} \right)^{-M-1} \le 3C\left(|p|+K\right)! \left(\frac{2}{\rho} \right)^{K+1} \left(\frac{k\rho}{2}\right)^{-|p|}.
\end{align*}
For the expansion \cref{eq-Bessel-type-expansion} with $(x,y)=(0,\rho)$ on $[k',\infty)$, by \cref{lemma-Jn-bound}, for each $p$,
\begin{equation*}
\begin{aligned}
&\int_{k'}^{\infty}\left|J_{p}(k\rho') e^{\i p \theta'} e^{-\sqrt{\lambda^2-k^2}\rho}\left(-\i w(\lambda)\right)^{p}f(\lambda)\right| d\lambda\\
\le & \text{ } \frac{1}{|p|!}\left(\frac{k\rho'}{2} \right)^{|p|}\cdot 3C\left(|p|+K\right)! \left(\frac{2}{\rho} \right)^{K+1} \left(\frac{k\rho}{2}\right)^{-|p|} \\
=& \text{ }3C\frac{(|p|+K)!}{|p|!}\left(\frac{2}{\rho} \right)^{K+1} \left(\frac{\rho'}{\rho}\right)^{|p|}.
\end{aligned}
\end{equation*}
Similar as in \cref{eq-disc-Bessel-expansion} and using the Fubini's theorem,
\begin{align*}
&\int_{k'}^{\infty} e^{-\sqrt{\lambda^{2}-k^{2}}(\rho+y')+\i
\lambda(-x')} f(\lambda)d\lambda \\
=& \int_{k'}^{\infty}{\sum_{p=-\infty}^{\infty}J_{p}(k\rho') e^{\i p \theta'} e^{-\sqrt{\lambda^2-k^2}\rho}\left(-\i w(\lambda)\right)^{p}f(\lambda) d\lambda} \\
=& \sum_{p=-\infty}^{\infty} J_{p}(k\rho') e^{\i p \theta'} E_{p}^+.
\end{align*}
When $P \ge (k\rho)^2/4+1-K$, the $P$-term truncation has the truncation error
\begin{equation*}
\left| \sum_{|p| \ge P} J_{p}(k\rho') e^{\i p \theta'} E_{p}^+ \right| \le \sum_{|p| \ge P}3C\frac{(|p|+K)!}{|p|!}\left(\frac{2}{\rho} \right)^{K+1} \left(\frac{\rho'}{\rho}\right)^{|p|} \le c(P)\left(\frac{\rho'}{\rho}\right)^{P}.
\end{equation*}
\end{proof}\begin{remark}
The bound of $|E_p^+|$ in \cref{lemma-Epp} is chosen as an analog of the asymptotic behavior of $H_n^{(1)}(x) \sim (n-1)! (x/2)^n/(\i\pi)$ for $x>0$ as $n\to\infty$ \cite[(9.3.1)]{abramowitz1966handbook}.
\end{remark}

\subsection{Convergence of general integral Bessel-type expansion}

In order to obtain the convergence estimate of the integral Bessel-type
expansion \cref{eq-Bessel-type-expansion} on an infinite interval, we will
follow two steps. First, the Cagniard--de Hoop transform \cite{chew} will be
used to convert the general $(x,y)$ case to the $(\rho,0)$ case as discussed
in \cref{lemma-Epp}, namely, the complex factor $e^{-\sqrt{\lambda^{2}-k^{2}%
}y+\i \lambda x}$ in \cref{eq-Bessel-type-expansion} is converted to $e^{-\sqrt
{\lambda^{2}-k^{2}}\rho}$; Second, we deform the new complex contour of
integration to the real axis, see the illustration in
\cref{fig:picchangingcontour}.

\subsubsection{The Cagniard--de Hoop transform}

Given $x$, $y$ and $T \in\mathbb{R}^{+}$ satisfying $x<Ty$. Let $(\rho
,\theta)$ be the polar coordinates of $(x,y)$. Let $\beta=\frac{\pi}{2}%
-\theta\in(0,\frac{\pi}{2})$, then $y+x\i= \rho e^{\i\beta}$.

Define the open set
\begin{equation}
\Omega=\left\{  z\in\mathbb{C}:\Re z>0,z\notin(0,k]\right\}  .
\end{equation}
Define the holomorphic Cagniard--de Hoop mapping $\phi:\Omega\rightarrow
\mathbb{C}$ by
\begin{equation}
\phi(z)=z\cos\beta+\i \sqrt{z^{2}-k^{2}}\sin\beta.
\end{equation}
For any positive real $w\in(k,+\infty)$, one can easily verify that an inverse
of $\phi$ at $w$ is found in the fourth quadrant as $\phi^{-1}(w)=w\cos
\beta-\i \sqrt{w^{2}-k^{2}}\sin\beta$. Let
\begin{equation}
\gamma^{+}=\phi\left(  (k,+\infty)\right)  \text{, }\gamma^{-}=\phi
^{-1}\left(  (k,+\infty)\right)  ,\label{def-gamma}%
\end{equation}
then, $\gamma^{+}\cup\gamma^{-}\cup\{k\cos\beta\}$ is the right branch of the
hyperbola
\begin{equation}
\Gamma=\left\{  a+b\i :a,b\in\mathbb{R},\frac{a}{\cos\beta}=\sqrt
{\frac{b^{2}}{\sin^{2}\beta}+k^{2}}\right\}
\end{equation}
with vertex $k\cos\beta$ on the real axis.
$\Gamma$ is known as the Cagniard--de Hoop contour.
The lines passing the origin with slope $\pm
\tan\beta$ are the asymptotes of $\gamma^\pm$.

Define regions to the right of $\Gamma$ in
the first and the fourth quadrant, respectively, by
\begin{equation}
D^{\pm}=\{z+t:z\in\gamma^{\pm},t\in\mathbb{R}^{+}\}.\label{def-Dpm}%
\end{equation}
$D^\pm$ are isomorphic as the following lemma shows.
\begin{lemma}\label{lemma-bijection}
$\phi|_{D^-}$ is a bijection from $D^-$ to $D^+$ with inverse $\phi^{-1}|_{D^+}$ given by
\begin{equation}
\phi^{-1}(w) = w\cos\beta-\i\sqrt{w^2-k^2}\sin\beta.
\end{equation}
\end{lemma}
\begin{proof}
See \cref{proof-lemma-bijection}.
\end{proof}

\begin{figure}[tbh]
\centering
\includegraphics[scale=0.7]{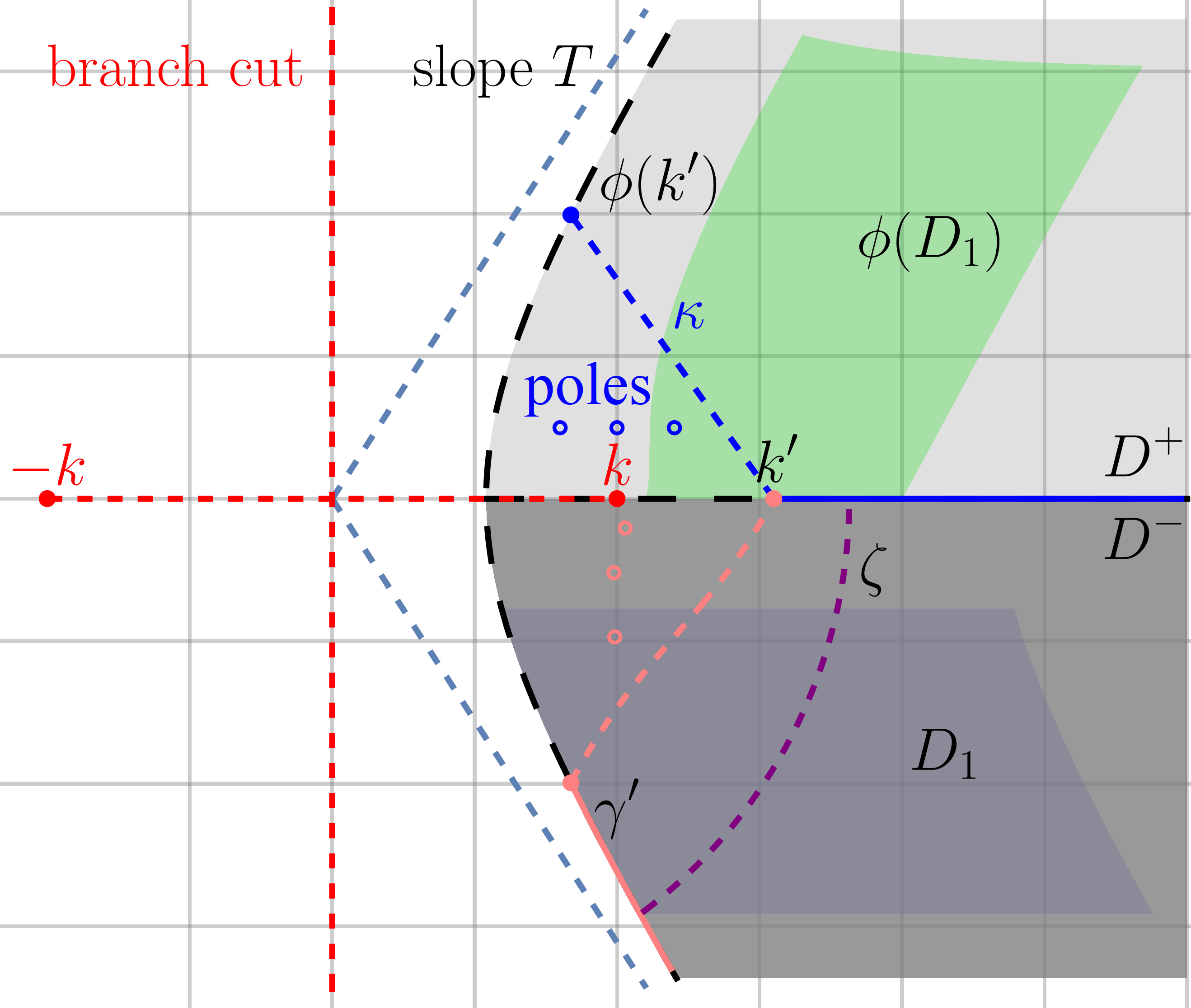}\caption{The
mapping $\phi$ which maps $D^{-}$ to $D^{+}$. The shadowed region $D_{1}$ is
for illustration.}%
\label{fig:picchangingcontour}%
\end{figure}

\subsubsection{The general Bessel-type expansion}

With the above preparation, we can now prove the expansion
\cref{eq-Bessel-type-expansion} when $f(\lambda)$ has a polynomial bound in
$\Omega$ when $|\lambda|$ is sufficiently large and $\Im\lambda/\Re\lambda$ is
bounded. To make it clear, we make the following assumptions.

\begin{assumption}\label{assumption}
Given $T>0$, $\epsilon_0 >0$.
Suppose $f(\lambda)$ is a complex function with branch points $\pm k_0,\cdots,\pm k_L$, it is even and is meromorphic in $\mathbb{C}$ excluding the branch cuts of $\sqrt{\lambda^2-k_l^2}$, $0 \le l \le L$, with poles of order up to one.
Assume that
\begin{itemize}
\item $f(\lambda)$ has a decomposition
\begin{equation}\label{eq-f-decomp-imag-poles}
f(\lambda) = \sum_{r=1}^{n_r}\frac{f_r}{\lambda-\lambda_r} +\bar{f}(\lambda)\text{, and } \bar{f}(\lambda) = \sum_{c=1}^{n_c}\frac{f_c}{\lambda-\lambda_c} +\bar{\bar{f}}(\lambda),
\end{equation}
here $\lambda_r \ne 0$ are all the real poles of $f(\lambda)$ with residue $f_r$, $\lambda_c$ are all the (complex) poles of $f(\lambda)$ in the region $\Omega_T^+ = \{ a+b\i: a>0, 0<b<aT \}$ with residue $f_c$, respectively, and are bounded by $\lambda_{M} = \max_{1\le c \le n_c}|\lambda_c|$.
\item $\left|\bar{\bar{f}}(\lambda)\right|\le C\left(1+|\lambda|^K\right)$ for any $\lambda\in\Omega_T^+\cup \mathbb{R}^+$ satisfying $ \max_{0\le l \le L}|\lambda-k_l|\ge\epsilon_0$, here $C>0$ and $K \in \mathbb{N}\cup\{0\}$ are given constants.
\item For $k'=4k_M+2\lambda_M+2\epsilon_0$, the integral $S=\int_{-k'}^{k'} \left|\bar{f}(\lambda)\right|d\lambda < +\infty$.
\end{itemize}
\end{assumption}

\begin{lemma}\label{lemma-xy}
Let $c_0 > 1$, $T>0$, $\epsilon_0 >0$ be some given constants.
Let $k_1,\cdots,k_L$ be given wave numbers of the layers with maximum $k_M$, and $k$ is one of the wave numbers.
Suppose the function ${f}(\lambda)$ is given satisfying \cref{assumption}, and $\bar{f}(\lambda)$ is so defined with the real poles removed from $f(\lambda)$, $(\rho,\theta)$ and $(\rho',\theta')$ are the polar coordinates of $(x,y)$ and $(x',y')$, respectively, and $\rho > c_0 \rho' \ge 0$.
Suppose $y>0$, $y+y'>0$ and $|x|<Ty$.
Then, the integral Bessel-type expansion \cref{eq-Bessel-type-expansion} holds (by replacing the original $f(\lambda)$) with $\bar{f}(\lambda)$ on the interval $(k',\infty)$, with the truncation error estimate for a finite $P$-term truncation
\begin{equation}
\left| \sum_{|p| \ge P} J_p(k\rho')e^{\i p \theta'} \int_{k'}^{\infty} \Psi(\lambda) \left(-\i w(\lambda)\right)^p \bar{f}(\lambda) d\lambda \right| \le c_+(P)\left( \frac{\rho'}{\rho} \right)^P
\end{equation}
for any sufficiently large $P \ge m_+(\rho)$, here $m_+(\rho)$ is an (at most) quadratic function of $\rho$, and $c_+(\cdot)$ is a function with polynomial growth rate.
\end{lemma}
\begin{proof}
If $x=0$, then $y = \sqrt{x^2+y^2} = \rho$, and
\begin{align*}
|\bar{f}(\lambda)| \le C\left(1+|\lambda|^K\right) + \sum_{c=1}^{n_c} \frac{|f_c|}{\Im \lambda_c} \le C_1\left(1+|\lambda|^K\right)
\end{align*}
for any $\lambda \in (k',\infty)$, here $C_1>0$ is a constant number.
By \cref{lemma-Epp}, we can choose
\begin{equation*}
m(\rho) = \left(\frac{k\rho}{2}\right)^{2} +1 -K,\quad
c_+(P)=6C_1(K+1)!\left( \frac{2c_0}{\rho (c_0-1)} \right)^{K+1}(P+K)^K.
\end{equation*}
If $x \ne 0$, without loss of generality assume $x>0$, since the $x<0$ case will follow by taking complex conjugates.
Let $\beta = \frac{\pi}{2}-\theta$, then $\tan\beta\in(0,T)$.
Let $\kappa$ be the segment from $\phi(k')$ to $k'$ (see \cref{fig:picchangingcontour}).
One can verify the length of $\kappa$ is bounded by $\sqrt{2}k'$, and that $ \lambda_M+k_M \le |\lambda| \le \sqrt{2}k'$ and $|\lambda - k_l| > \epsilon_0$ for $\forall\lambda \in \kappa$ and $\forall 0 \le l \le L$.
Define
\begin{equation}
E = \int_{\kappa \cup (k',\infty)}\Psi(\lambda)\Psi'(\lambda)\bar{f}(\lambda) d\lambda,\quad G=\int_{\kappa}\Psi(\lambda)\Psi'(\lambda)\bar{f}(\lambda) d\lambda,
\end{equation}
we will discuss the expansions on $E$ and on $G$, separately, then give the integral Bessel-type expansion for $E-G$.
For $G$, on $\kappa$ we have the bound of $\bar{f}(\lambda)$ by
\begin{equation}
|\bar{f}(\lambda)| \le C\left(1+|\lambda|^K\right) + \sum_{n=1}^{n_c}\frac{|f_c|}{|\lambda-\lambda_c|} \le C\left(1+(\sqrt{2}k')^K\right) + \sum_{c=1}^{n_c}\frac{|f_c|}{k_M} := C_2,
\end{equation}
so by \cref{lemma-bounded-contour}, the Bessel-type expansion on $\kappa$ is given by
\begin{equation}
G = \sum_{p=-\infty}^{\infty}J_p(k\rho')e^{\i p \theta'} G_p,\quad G_p = \int_{\kappa} \Psi(\lambda)\left(-\i w(\lambda)\right)^p \bar{f}(\lambda)d\lambda,
\end{equation}
with truncation error
\begin{equation}\label{eq-truncation-K}
\left|\sum_{|p| \ge P}J_p(k\rho')e^{\i p \theta'} G_p\right| \le c_\kappa(P)\left( \frac{\rho'}{\rho} \right)^P \text{ for } P \ge m_\kappa(\rho),
\end{equation}
here $c_\kappa(P) = {2c_0}C_2\cdot \sqrt{2}k'/(c_0-1)$, $m_\kappa(\rho) = e(\lambda_M+k/2)\rho$.
For the contour $\kappa \cup (k',\infty)$, with the substitution $\lambda = \phi(\lambda')$
we have
\begin{equation*}
\sqrt{\lambda'^2-k^2}=\frac{\phi(\lambda')-\lambda'\cos\beta}{\i\sin\beta}=\frac{\lambda-(\lambda\cos\beta-\i\sqrt{\lambda^2-k^2}\sin\beta)\cos\beta}{\i\sin\beta},
\end{equation*}
so $\Psi(\lambda) = e^{-\sqrt{\lambda'^2-k^2}\rho}\text{, }\Psi'(\lambda) =e^{-\sqrt{{\lambda'}^2-k^2}\rho' \sin(\theta'-\beta) -\i\lambda' \rho' \cos(\theta'-\beta)}, w(\lambda') = e^{\i\beta}w(\lambda)$.
Hence
\begin{equation}
E = \int_{\phi^{-1}(\kappa)\cup\gamma'} e^{-\sqrt{{\lambda'}^2-k^2}(\rho + \rho' \sin(\theta'-\beta)) -\i\lambda' \rho' \cos(\theta'-\beta)} \tilde{f}(\lambda')d\lambda',
\end{equation}
here $\gamma' = \phi^{-1}\left( (k',\infty) \right)$ is the lower part of $\gamma$ starting from $\phi^{-1}(k')$, and
\begin{equation}\label{eq-in-thm-fprime}
\tilde{f}(\lambda') = \bar{f}(\lambda) \frac{d\lambda}{d\lambda'}=\bar{f}\left(\phi(\lambda')\right) \frac{\sqrt{\phi(\lambda')^2-k^2}}{\sqrt{\lambda'^2-k^2}}.
\end{equation}
Since $\phi(\lambda')$ has a polynomial bound, roughly,
\begin{equation}\label{eq-in-thm-pbound-phi}
\left|\phi(\lambda')\right| = \left| \lambda'\cos\beta+\i\sqrt{\lambda'^2-k^2}\sin\beta \right| \le |\lambda'|+\sqrt{|\lambda'|^2+k^2} \le 2|\lambda'|+k,
\end{equation}
when $\lambda' \in D^-$ and $|\lambda'|$ is sufficiently large, $\tilde{f}(\lambda')$ also has a polynomial bound of $|\lambda'|$.

Next, we proceed to changing the contour of the integral $E$ from $\phi^{-1}(\kappa)\cup\gamma'$ to $(k',\infty)$.
Let $\zeta$ be the counterclockwise arc with radius $r$ connecting $\phi^{-1}(\kappa)\cup\gamma'$ and the real axis, parameterized by $\lambda'=re^{\i\eta}$, the range of $\eta$ is a subset of $(-\beta, 0)$.
On the arc $\zeta:\lambda' = re^{\i\eta}$, as $r \to \infty$ 
the exponent of the integrand in $E$ satisfies
\begin{align*}
-\sqrt{\lambda'^2-k^2}\left( \rho+\rho'\sin(\theta'-\beta) \right)-\i\lambda'\rho'\cos(\theta'-\beta) &\sim -\lambda'e^{\i\beta}\left( \rho e^{-\i\beta} + \i \rho'e^{-\i\theta'} \right) \\
&\sim r \bar{\rho}\exp\left(\i\left( \eta+\bar{\theta} + \beta +\frac{\pi}{2}\right)\right)
\end{align*}
where $(\bar{\rho},\bar{\theta})$ are the polar coordinates of $(x-x',y+y')$, and the rest of the integrand has a polynomial bound. 
Since $y+y'>0$, $\rho>\rho'$, one can verify $\bar{\theta} \in (0,\pi - \beta)$.
Then
\begin{equation*}
\Re \left\{ r \bar{\rho}\exp\left(\i\left( \eta+\bar{\theta} + \beta +\frac{\pi}{2}\right)\right) \right\} \le r\cdot \max\left\{ -(y+y'),-\rho-\rho'\sin(\theta'-\beta) \right\}
\end{equation*}
for any $\eta \in (-\beta, 0)$, so the integrand on $\zeta$ decays exponentially, and the corresponding integral on $\zeta$ vanishes as $r \to +\infty$.
Also notice that there are no poles of $\tilde{f}(\lambda')$ in $D' \subset D^-$ which is the region enveloped by $(k',+\infty)$ and $\phi^{-1}(\kappa) \cup \gamma'$, since $\phi$ is a holomorphic function on $D'$, and there are no poles in $D'$ because for any $\lambda' \in D'$ and any pole $\lambda_c \in \Omega_T^+$, $|\phi(\lambda')| \ge \lambda_M+k_M > |\lambda_c|$.
Hence, by deforming the integration contour in $E$ to the real axis, we have
\begin{equation}\label{eq-in-thm-E-w-poles}
E = E' := \int_{(k',\infty)} e^{-\sqrt{{\lambda'}^2-k^2}(\rho + \rho' \sin(\theta'-\beta)) -\i\lambda' \rho' \cos(\theta'-\beta)} \tilde{f}(\lambda')d\lambda'.
\end{equation}
For $\lambda' \in (k',\infty)$, recall that
\begin{equation*}
\tilde{f}(\lambda) = \frac{\sqrt{\lambda^2-k^2}}{\sqrt{\lambda'^2-k^2}}\left( \bar{\bar{f}}(\phi(\lambda')) + \sum_{c=1}^{n_c} \frac{f_c}{\phi(\lambda')-\lambda_c} \right),
\end{equation*}
for each $\lambda_c$ we have $|\phi(\lambda')-\lambda_c| \ge |\phi(\lambda')|-|\lambda_c| \ge \sqrt{\lambda'^2-k^2\sin^2\beta} - \lambda_M \ge 3k_M$, so using \cref{eq-in-thm-pbound-phi}, there exists some constant $C_2>0$ such that
\begin{equation*}
|\tilde{f}(\lambda')| \le \frac{2|\lambda'|+2k}{3k_M} \left( C\left( 1+(2|\lambda'|+k)^K \right) + \sum_{c=1}^{n_c}\frac{|f_c|}{3k_M} \right) \le C_2 |\lambda'|^{K+1}.
\end{equation*}
Hence by \Cref{lemma-Epp}, $E'$ has the series expansion
\begin{equation}\label{eq-in-thm-Ep-expansion}
E' = \sum_{p=-\infty}^{\infty}J_p(k \rho')e^{\i p (\theta'-\beta)} \int_{k'}^{\infty}e^{-\sqrt{\lambda'^2-k^2}\rho} \left(-\i w(\lambda') \right)^p \tilde{f}(\lambda')d\lambda'
\end{equation}
with a $P$-term truncation error
\begin{equation}\label{eq-in-thm-Ep-err}
\left| \sum_{|p| \ge P}J_p(k \rho')e^{\i p (\theta'-\beta)} \int_{k'}^{\infty}e^{-\sqrt{\lambda'^2-k^2}\rho} \left(-\i w(\lambda') \right)^p \tilde{f}(\lambda')d\lambda' \right| \le c_{E'}(P) \left(\frac{\rho'}{\rho} \right)^{P}
\end{equation}
for $P \ge m_{E'}(\rho)= (k\rho)^2/4 -K$, here
\begin{equation}
c_{E'}(P) = 6C_2(K+2)!\left( \frac{2c_0}{\rho (c_0-1)} \right)^{K+2}(P+K+1)^{K+1}.
\end{equation}
In the series \cref{eq-in-thm-Ep-expansion}, the $p-$th term is
\begin{align}
&e^{-\i p\beta} \int_{(k,\infty)}e^{-\sqrt{\lambda'^2-k^2}\rho} \left(-\i w(\lambda') \right)^p \tilde{f}(\lambda')d\lambda' \nonumber \\
={}& e^{-\i p\beta} \int_{\phi^{-1}(\kappa) \cup \gamma'}e^{-\sqrt{\lambda'^2-k^2}\rho} \left(-\i w(\lambda') \right)^p \tilde{f}(\lambda')d\lambda' \\
={}& \int_{\kappa \cup (k',\infty)}\Psi(\lambda) \left(-\i w(\lambda) \right)^p \bar{f}(\lambda)d\lambda := E_p,
\end{align}
here the first equality is derived similarly by changing the contour of the integral, where on the path $\zeta: \lambda' = re^{\i \eta}$ the integrand decays exponentially as $r \to \infty$, since the real part of the exponent
\begin{equation*}
\Re{\left( -\sqrt{(re^{\i\eta})^2-k^2}\rho \right) } \sim \Re(-re^{\i\eta}\rho) \le -ry,
\end{equation*}
while the remaining parts have polynomial growth rate.
The second equality is by the substitution from $\lambda'$ to $\lambda$.
In total we have proved the series expansion of $E$ given by
$
E = \sum_{p=-\infty}^{\infty}J_{p}(k\rho')e^{\i p\theta'}E_p
$ 
with a $P$-term truncation error
\begin{equation}\label{eq-truncation-E}
\left|E-\sum_{|p|<P}J_{p}(k\rho')e^{\i p\theta'}E_p\right| \le c_{E'}(P) \left( \frac{\rho'}{\rho} \right)^P \text{ for }P \ge m_{E'}(\rho).
\end{equation}
For each $p$,
\begin{align}
\begin{split}
E_p - G_p &= \int_{\kappa \cup (k',\infty)}\Psi(\lambda) \left(-\i w(\lambda) \right)^p \bar{f}(\lambda)d\lambda - \int_{\kappa}\Psi(\lambda) \left(-\i w(\lambda) \right)^p \bar{f}(\lambda)d\lambda \\
&= \int_{k'}^{\infty}\Psi(\lambda) \left(-\i w(\lambda) \right)^p \bar{f}(\lambda)d\lambda,
\end{split}
\end{align}
which is the desired expansion function in the Bessel-type expansion \cref{eq-Bessel-type-expansion}.
Hence by combining the results \cref{eq-truncation-K} and \cref{eq-truncation-E}, for any (finite) $P \ge \max\{ m_{E'}(\rho), m_\kappa(\rho) \}$,
\begin{align*}
\begin{split}
&\left|\int_{k'}^{\infty}\Psi(\lambda)\Psi'(\lambda) \bar{f}(\lambda)d\lambda -\sum_{|p| < P} J_{p}(k\rho')e^{\i p\theta'}\int_{k'}^{\infty}\Psi(\lambda) \left(-\i w(\lambda) \right)^p \bar{f}(\lambda)d\lambda \right| \\
\le & \left|E-\sum_{|p| < P} J_p(k\rho')e^{\i p\theta'} E_p\right| + \left|G-\sum_{|p| < P} J_p(k\rho')e^{\i p\theta'}G_p\right| \le \left( c_{E'}(P) + c_{\kappa}(P)\right)\left( \frac{\rho'}{\rho} \right)^P
\end{split}
\end{align*}
which suggests $c_+(P) = c_{E'}(P) + c_{\kappa}(P)$ and $m_+(\rho) =\max\{ m_{E'}(\rho), m_\kappa(\rho) \}$.
\end{proof}

\begin{theorem}[the Bessel-type expansion]\label{thm-Bessel-type-expansion}Suppose conditions of \cref{lemma-xy} are satisfied.
Further suppose $0<\rho_m<\rho_M$ are given such that $\rho \in [\rho_m,\rho_M]$.
Then, the integral Bessel-type expansion \cref{eq-Bessel-type-expansion} holds with a truncation error estimate
\begin{equation}\label{eq-thm-truncation-result}
\left|\int_{-\infty}^{\infty}e^{-\sqrt{\lambda^{2}-k^{2}}(y+y')+\i \lambda (x-x')}f(\lambda)d\lambda
- \sum_{|p|<P}J_{p}(k\rho')e^{\i p\theta'}F_p\right|\le c(P) \left(\frac{\rho'}{\rho}\right)^P
\end{equation}
when $P$ is sufficiently large and $c(\cdot)$ is a function with a polynomial growth rate.
\end{theorem}
\begin{proof}
Consider the decomposition of the integral
\begin{equation}\label{eq-in-thm-int-aligned}
\begin{split}
I=&\int_{-\infty}^{\infty}e^{-\sqrt{\lambda^{2}-k^{2}}(y+y')+\i \lambda (x-x')}f(\lambda)d\lambda \\
=& \sum_{r=1}^{n_r}\tau_r \i\pi \Psi(\lambda_r)\Psi'(\lambda_r)f_r +  \left(\int_{-\infty}^{-k'}+\int_{-k'}^{k'}+\int_{k'}^{\infty}\right) \Psi(\lambda)\Psi'(\lambda)\bar{f}(\lambda)d\lambda \\
:=&\sum_{r=1}^{n_r}I_r + I_- + I_0 + I_+,
\end{split}
\end{equation}
here each $\tau_r = \pm 1$ are given by the well-posed physical problem (see \cref{eq-int-pole}).
Each term of the decomposition with index $j$ has the corresponding Bessel-type expansion, $j=0,1,\cdots,n_r,+,-$.
Namely, for each $I_r$, by \cref{lemma-pole}, by choosing $c_r(P) = {2\pi c_0 |f_r|}/{(c_0-1)}$ and $e( |\lambda_r|+{k}/{2} )\rho$,
the pointwise Bessel-type expansion \cref{eq-Bessel-type-expansion-discrete} holds
\begin{align*}
I_r = \sum_{p=-\infty}^{\infty} J_p(k \rho') e^{\i p \theta'} I_{r,p},\quad I_{r,p} = \tau_r \i \pi \Psi(\lambda_r ) \left( -\i w(\lambda_r) \right)^p
\end{align*}
with the truncation error for a $P$-term truncation
\begin{equation}
\left| \sum_{|p| \ge P} J_p(k \rho')e^{\i p \theta'} I_{r,p} \right| \le c_r(P) \left(\frac{\rho'}{\rho}\right)^P \text{ for } P \ge m_r(P).
\end{equation}
For $I_0$, by \cref{lemma-bounded-interval}, by choosing $c_0(P)={2\pi c_0 S}/{(c_0-1)}$ and $m_0(\rho) = ek'\rho$, the integral Bessel-type expansion \cref{eq-Bessel-type-expansion} holds
\begin{align*}
I_0 = \sum_{p=-\infty}^{\infty} J_p(k \rho') e^{\i p \theta'} I_{0,p},\quad I_{0,p} = \int_{-k'}^{k'} \Psi(\lambda ) \left( -\i w(\lambda) \right)^p \bar{f}(\lambda) d\lambda
\end{align*}
with the truncation error for a $P$-term truncation
\begin{equation}
\left| \sum_{|p| \ge P} J_p(k \rho')e^{\i p \theta'} I_{0,p} \right| \le c_0(P) \left(\frac{\rho'}{\rho}\right)^P \text{ for } P \ge m_0(P).
\end{equation}
For $I_+$ and $I_-$, by choosing the $c_+(P)$ and $m_+(\rho)$ provided by \cref{lemma-xy}, and $c_-(P)=c_+(P)$ and $m_-(\rho)=m_+(\rho)$ due to the symmetry, the integral Bessel-type expansion \cref{eq-Bessel-type-expansion} holds $I_{\pm} = \sum_{p=-\infty}^{\infty} J_p(k \rho') e^{\i p \theta'} I_{\pm,p}$,
where
\begin{align*}
I_{+,p} = \int_{k'}^{\infty} \Psi(\lambda ) \left( -\i w(\lambda) \right)^p \bar{f}(\lambda) d\lambda, \quad I_{-,p} = \int_{-\infty}^{-k'} \Psi(\lambda ) \left( -\i w(\lambda) \right)^p \bar{f}(\lambda) d\lambda
\end{align*}
with the truncation error for a $P$-term truncation
\begin{equation}
\left| \sum_{|p| \ge P} J_p(k \rho')e^{\i p \theta'} I_{\pm,p} \right| \le c_{\pm}(P) \left(\frac{\rho'}{\rho}\right)^P \text{ for } P \ge m_{\pm}(P).
\end{equation}
For each $p$, the expansion functions add up to $F_p$ because
\begin{align*}
\begin{split}
F_p={}& \sum_{r=1}^{n_r}\tau_r \i\pi \Psi(\lambda_r)\left(-\i w(\lambda_r)\right)^p f_r +  \left(\int_{-\infty}^{-k'}+\int_{-k'}^{k'}+\int_{k'}^{\infty}\right) \Psi(\lambda)\left(-\i w(\lambda)\right)^p\bar{f}(\lambda)d\lambda \\
={}&\sum_{r=1}^{n_r}I_{r,p} + I_{-,p} + I_{0,p} + I_{+,p}.
\end{split}
\end{align*}
Hence by adding the series expansions up, for any (finite) $ P \ge m(\rho):= \max_j m_j(\rho)$,
\begin{equation*}
\left|I - \sum_{|p|<P}J_p(k \rho')e^{\i p \theta'} F_p\right| \le \sum_{j} \left|I_j-\sum_{|p|< P}J_p(k \rho')e^{\i p \theta'} I_{j,p}\right| \le \sum_{j} c_j(P) \left( \frac{\rho'}{\rho} \right)^P.
\end{equation*}
Therefore the Bessel-type expansion holds for any $P \ge m(\rho)$, and the truncation error is bounded by $\sum_j c_j(P)(\rho'/\rho)^P$.
Since the only dependence of $c_j(P)$ on $\rho$ appears in the terms $c_{\pm}(P)$ which reach upper bounds as $\rho\to \rho_m$, and each $m_j(\cdot)$ is an increasing function, we conclude that by choosing $c(P) = \left. \sum_j c_j(P)\right|_{\rho = \rho_m}$, the truncation error estimate \cref{eq-thm-truncation-result} for any $P \ge m(\rho_M)$.
\end{proof}

\subsection{Proof of Theorem \ref{thm-exp-conv}}

Here, only the proof of the ME \cref{eq-err-ME} will be given as the others
can be similarly treated.

Let $\tilde{x}=x-x_{c}$, $\tilde{y}=\tau^{\ast}(y-d_{t}^{\ast})+\tau^{\star
}(y_{c}-d_{s}^{\star})$, $\tilde{x}^{\prime}=x^{\prime}-x_{c}$, $\tilde
{y}^{\prime}=\tau^{\star}(y^{\prime}-y_{c})$, and
\begin{align}
f(\lambda)=e^{(\sqrt{\lambda^{2}-k_{s}^{2}}-\sqrt{\lambda^{2}-k_{t}^{2}}%
)\tau^{\ast}(y-d_{t}^{\ast})}\sigma_{ts}^{\ast\star}(\lambda)
\end{align}
so that the integral \cref{eq-urf-decomp-int} can be written as
\[
\begin{aligned} u_{ts}^{\ast\star}\left( \mathbf{x}, \mathbf{x}'; \sigma_{ts}^{\ast \star} \right)&= \int_{-\infty}^{\infty}e^{-\sqrt{\lambda^2-k_t^2}\tau^{\ast}(y-d_t^{\ast})-\sqrt{\lambda^2-k_s^2}\tau^{\star}(y'-d_s^{\star})+\i\lambda(x-x')}\sigma_{ts}^{\ast \star}(\lambda)d\lambda \\ &= \int_{-\infty}^{\infty} e^{-\sqrt{\lambda^{2}-k_s^{2}}(\tilde{y}+\tilde{y}')+\i \lambda(\tilde{x}-\tilde{x}')} f(\lambda)d\lambda. \end{aligned}
\]
With the assumption that the sources, the targets and the centers are bounded
in a given box, and that $|y_{c}-d_{s}^{\star}|$ has a nonzero lower bound,
there exists fixed $T>0$ such that $|\tilde{x}|<T\tilde{y}$. By
\cref{theorem-poly-bound}, $\sigma_{ts}^{\ast\star}(\lambda)$ has a polynomial
bound in the region $\Omega_{T}=\left\{  a+b\i:a>0,-aT<b<aT\right\}  $ when
$\Re\lambda$ is sufficiently large, which easily implies the same for
$f(\lambda)$. With the decomposition \cref{eq-f-decomp-imag-poles}, when
neighborhoods of each branch point $k_{l}$ with a sufficiently small radius
$\epsilon_{0}>0$ are excluded from $\Omega_{T}$, $\bar{\bar{f}}(\lambda)$ is
finite and hence has polynomial bound. Replacing $x,y,x^{\prime},y^{\prime},k
$ in \cref{thm-Bessel-type-expansion} by $\tilde{x},\tilde{y},\tilde
{x}^{\prime},\tilde{y}^{\prime},k_{s}$ finishes the proof of \cref{eq-err-ME}.

For the LE \cref{eq-err-LE}, similarly, choose $\tilde{x} = x_{c}^{l} -
x^{\prime}$, $\tilde{y} = \tau^{\ast}(y_{c}^{l} - d_{t}^{\ast})+\tau^{\star
}(y^{\prime}-d_{s}^{\star})$, $\tilde{x}^{\prime}=x_{c}^{l}-x$, $\tilde
{y}^{\prime}= \tau^{\ast}(y-y_{c}^{l})$, $k=k_{t}$ and $f(\lambda) =
e^{(\sqrt{\lambda^{2}-k_{t}^{2}}-\sqrt{\lambda^{2}-k_{s}^{2}})\tau^{\star
}(y^{\prime}-d_{s}^{\star})}\sigma_{ts}^{\ast\star}(\lambda)$.

For the M2L \cref{eq-err-M2L}, for each LE coefficient $L_{m}^{\ast\star
}(\mathbf{x}_{c}^{l},\mathbf{x})$,
choose $\tilde{x}=x_{c}^{l}-x_{c}$, $\tilde{y}=\tau^{\ast}(y_{c}^{l}%
-d_{t}^{\ast})+\tau^{\star}(y_{c}-d_{s}^{\star})$, $\tilde{x}^{\prime
}=x^{\prime}-x_{c}$, $\tilde{y}^{\prime}=\tau^{\star}(y^{\prime}-y_{c})$,
$k=k_{s}$ and
\[
f(\lambda) = e^{(\sqrt{\lambda^{2}-k_{s}^{2}}-\sqrt{\lambda^{2}-k_{t}^{2}%
})\tau^{\ast}(y_{c}^{l}-d_{t}^{\ast})}\sigma_{ts}^{\ast\star}(\lambda) \left(
\i w_{t}(\lambda)^{-1}\right) ^{m}.
\]

For the L2L \cref{eq-err-L2L}, for each LE coefficient $L_{m}^{\ast\star
}(\mathbf{x}_{c}^{l},\mathbf{x})$, choose $\tilde{x}=x_{c}^{l}-x^{\prime}$,
$\tilde{y}=\tau^{\ast}(y_{c}^{l}-d_{t}^{\ast})+\tau^{\star}(y^{\prime}%
-d_{s}^{\star})$, $\tilde{x}^{\prime}=x_{c}^{l}-\tilde{x}_{c}^{l}$, $\tilde
{y}^{\prime}=\tau^{\star}(\tilde{y}_{c}^{l}-y_{c}^{l})$, $k=k_{t}$ and
\[
f(\lambda) = e^{(\sqrt{\lambda^{2}-k_{t}^{2}}-\sqrt{\lambda^{2}-k_{s}^{2}%
})\tau^{\star}(y^{\prime}-d_{s}^{\star})}\sigma_{ts}^{\ast\star}(\lambda)
\left( \i w_{t}(\lambda)^{-1}\right) ^{m}.
\]

\section{Conclusion}

\label{sect-4} Far-field expansions of ME, LE as well as M2L and L2L
translation operators are derived and the exponential convergence rates are
proven. The analysis shows the convergence of ME and LE for the reaction field
components depends on a polarized distance between the target and the
polarized image of the sources. This fact shows how the ME and LE can be used
in the traditional FMM framework, which has been implemented in the 3-D case
in \cite{bo2019sphfmm}.

In a future work, we will extend the convergence result to 3-D Helmholtz
equations in layered media.

\appendix


\section{Proof of Lemma \ref{lemma-bijection}}

\label{proof-lemma-bijection} We begin with the following two lemmas, which
are stated given the same conditions as in \Cref{lemma-bijection}.
\begin{lemma}\label{lemma-2quadI}
Let $a,b\in\mathbb{R}$ such that $z=a+b\i\in D^{-}$, then $\Re \phi(z)>0$, $\Im\phi(z)>0$.
\end{lemma}
\begin{proof}
Let $u,v\in\mathbb{R}$ such that $u+v\i=\sqrt{z^{2}-k^{2}}$, then $uv=ab<0$.
With the convention of the branch cut \cref{eq-branch-cut}, we have $u>0$, so $v<0$.
Recall that $\beta \in (0,\frac{\pi}{2})$, we have $u\sin\beta-b\cos\beta > 0$ and $\Re \phi(z) = a \cos \beta - v \sin\beta >0$.
For $\Im \phi(z)$, let
\begin{equation}
Q_{1}=(a^{2}-b^{2}-k^{2})^{2}+4a^{2}b^{2}, \quad Q_{2}=(a^{2}-b^{2}-k^{2})\sin^{2}\beta-2b^{2}\cos^{2}\beta.
\end{equation}
By simple calculation, we have $2u^{2}\sin^{2}\beta-2b^{2}\cos^{2}\beta=\sqrt{Q_{1}}\sin^{2}\beta+Q_{2}$, and
\begin{equation*}
Q_{1}\sin^{4}\beta-Q_{2}^{2}=b^{2}\sin^{2}(2\beta)\left(\frac{a^{2}}{\cos^{2}\beta}-\frac{b^{2}}{\sin^{2}\beta}-k^{2}\right)>0,
\end{equation*}
so $\sqrt{Q_{1}}\sin^{2}\beta=\left\vert \sqrt{Q_{1}}\sin^{2}\beta\right\vert>|Q_{2}|$,
which implies
\begin{equation*}
\Im \phi(z) = b\cos\beta + u \sin\beta = \frac{\sqrt{Q_{1}}\sin^{2}\beta+Q_{2}}{2(u\sin\beta-b\cos\beta)}>0.
\end{equation*}
\end{proof}
\begin{lemma}\label{lemma-noGamma}
If $w\in\gamma^+$, then $\phi(z)\ne w$ for any $ z\in D^-$.
\end{lemma}
\begin{proof}
Suppose for contradiction that $z\in D^{-}$, $\phi(z)=w$. Since $w\in\gamma^+$, there exists a positive real number $x_0\ge k$ such that $w=x_0\cos\beta+\i\sqrt{x_0^{2}-k^{2}}\sin\beta$.
Therefore, $x_0$ and $z$ are distinct roots of the quadratic equation
\begin{align*}
\lambda^{2}-2\lambda w\cos\beta+w^{2}=k^{2}\sin^{2}\beta
\end{align*}
of $\lambda$.
Hence $z=2w\cos\beta-x_0=x_0\cos(2\beta)+\i\sqrt{x_0^{2}-k^{2}}\sin(2\beta)$ has nonnegative imaginary part, which contradicts the assumption that $z\in D^{-}$.
\end{proof}
\begin{proof}[Proof of \cref{lemma-bijection}]
Define $\phi' : D^+ \to \mathbb{C}$ by
$
\phi'(w)=w\cos\beta-\i\sqrt{w^2-k^2}\sin\beta.
$ 
It suffices to show $\phi'$ is the inverse of $\phi$ on ${D^+}$, i.e. $\phi^{-1}|_{D^+} = \phi'$.
First, we will show that $\phi(D^{-})\subset D^{+}$.
By \cref{lemma-2quadI} and \cref{lemma-noGamma}, $\phi(D^{-})$ is a subset of the first quadrant, and it has no intersection with the hyperbola $\Gamma$.
If $w=\phi(z)$ for some $z\in D^{-}$ and $w\notin D^{+}$, when we move $z$ horizontally to the left, eventually $z$ touches $\Gamma$ and $\phi(z)$ approaches the positive real axis, so the trajectory of $\phi(z)$, which must be continuous because $\phi$ is holomorphic, crosses $\Gamma$ in the first quadrant, but it contradicts with \cref{lemma-noGamma} since the intersection must has its inverse in $D^{-}$.
Similarly (by taking complex conjugates), $\phi^{\prime}(D^{+})\subset D^{-}$.
Second, we will show that $\phi$ is bijective on $D^{-}$ with inverse $\phi^{\prime}$.
Let $a,b\in\mathbb{R}^{+}$ such that $z=a+b\i\in D^{-}$, then $w=\phi(z)\in D^{+}$ is one of the roots of the quadratic equation of $\lambda$
\begin{equation}
\lambda^{2}-2\lambda z\cos\beta+z^{2}=k^{2}\sin^{2}\beta. \label{wzeqn}%
\end{equation}
Let $u,v\in\mathbb{R}$ such that $\sqrt{z^{2}-k^{2}}=u+v\i$, then $u>0$, the pair of roots are given by
\begin{equation}
\lambda_{\pm} =(a\cos\beta\mp v\sin\beta)+\i(b\cos\beta\pm u\sin\beta).
\end{equation}
By \cref{lemma-2quadI}, $\Im w = \Im\phi(z) >0$, so $w = \lambda^+$.
Conversely, $z$ is the only root of the quadratic equation
\begin{equation*}
\lambda^{2}-2\lambda w\cos\beta+w^{2}=k^{2}\sin^{2}\beta
\end{equation*}
in $D^{-}$ provided $\phi(z)=w$ by the similar reason, so $\phi$ is injective and $z=\phi^{\prime}(w)$.
Repeat this step for any $w^{\prime}\in D^{+}$ and let $z^{\prime}=\phi^{\prime}(w^{\prime})$, we have $\phi$ is surjective and $w^{\prime}=\phi(\phi^{\prime}(w^{\prime}))$.
\end{proof}


\section{Properties of the Green's function in layered media}

\label{Appendix-sigma}
As the preliminaries of the proofs of the convergence estimates, some
properties of the Green's function in layered media are discussed, including
the algebraic structure of the reflection/transmission coefficients
$\sigma_{ts}^{\ast\star}(\lambda)$, and its polynomial bound.

\subsection{The algebraic structure of the reflection/transmission
coefficient}

\label{section-decomp-rf} In paper \cite{bo2018taylorfmm} we conclude that the
reaction field $u^{\mathrm{r}}$ has a decomposition \cref{eq-urf-decomp}. To
solve the coefficients $\sigma_{ts}^{\ast\star}(\lambda)$, the interface
conditions deserve some further observation.

Each interface equation at $y=d_{l}$ given by \cref{eq-interface-cond-form} is
equivalent to
\begin{equation}
\left[  a_{t}\hat{u}^{\text{r}}+b_{t}\frac{\partial\hat{u}^{\text{r}}%
}{\partial y}\right]  =-\left[  \delta_{t,s}\left(  a_{t}\hat{G}^{\text{f}%
}-b_{t}\frac{\partial\hat{G}^{\text{f}}}{\partial y^{\prime}}\right)  \right]
\text{, at }y=d_{l}\label{eq-interface-cond-freq-domain}%
\end{equation}
in the frequency domain. With the conventions $d_{-1}=+\infty$ and
$d_{L}=-\infty$, and the decomposition of $u^{\mathrm{r}}$
\cref{eq-urf-decomp} introduced, by a well separation of variables $y$ and
$y^{\prime}$, the above equation can be further expanded as linear equations
of $\sigma_{ls}^{\ast\star}(\lambda)$ and $\sigma_{l+1,s}^{\ast\star}%
(\lambda)$:
\begin{align}
\label{eq-interface-cond-expanded}-c_{l}^{-}\sigma_{ls}^{\uparrow\star} -
c_{l}^{+} e_{l}\sigma_{ls}^{\downarrow\star} + c_{l+1}^{-}e_{l+1}%
\sigma_{l+1,s}^{\uparrow\star} + c_{l+1}^{+}\sigma_{l+1,s}^{\downarrow\star} =
v_{l,s}^{\star},\quad\star\in\{\uparrow,\downarrow\}
\end{align}
here $v_{l,s}^{\uparrow} = \delta_{l,s}c_{l}^{+}/(4\pi h_{l})$, $v_{l,s}%
^{\downarrow} = -\delta_{l+1,s}c_{l+1}^{-}/(4\pi h_{l+1})$, and the
coefficients
\begin{equation}
h_{t}=\sqrt{\lambda^{2}-k_{t}^{2}}\text{, }c_{t}^{\pm}=a_{t}\pm b_{t}%
h_{t}\text{, }e_{t}=e^{-h_{t}(d_{t-1}-d_{t})}\text{, }
t=l,l+1\label{def-field-gen-coeff}%
\end{equation}
Each $e_{t}$ vanishes in \cref{eq-interface-cond-expanded} if and only if
$t=0$ or $t=L$, corresponding to a prohibited propagating direction of $\ast$,
where $\sigma_{ts}^{\ast\star}(\lambda)=0$ in such case, and the term can be
safely neglected from the equations.

If we expand all the $2L$ interface conditions into the form
\cref{eq-interface-cond-expanded}, two linear system of unknowns
$\boldsymbol{\sigma}_{s}^{\uparrow}$ which consists of components $\sigma
_{ts}^{\ast,\uparrow}$ and $\boldsymbol{\sigma}_{s}^{\downarrow}$ which
consists of components $\sigma_{ts}^{\ast,\downarrow}$ are then derived in the
form
\begin{equation}
\label{eq-LS-sigma}\mathbf{A}(\lambda)\boldsymbol{\sigma}_{s}^{\uparrow
}(\lambda)=\mathbf{b}_{s}^{\uparrow}(\lambda)\text{, }\mathbf{A}%
(\lambda)\boldsymbol{\sigma}_{s}^{\downarrow}(\lambda)=\mathbf{b}%
_{s}^{\downarrow}(\lambda),
\end{equation}
here $\mathbf{A}$ does not depend on the source layer $s$ or the
source-induced direction $\star$. The functions $\sigma_{ts}^{\ast\star
}(\lambda)$ can be solved from linear systems \cref{eq-LS-sigma} using
Cramer's rule, so the complex roots of $\det\mathbf{A}(\lambda)$ are the
common poles of each $\sigma_{ts}^{\ast\star}(\lambda)$.

Now consider the field $\mathbb{F}$ of some functions of $\lambda$ defined by
field extension from $\mathbb{C}$
\begin{equation}
\label{def-field-F}\begin{aligned} \mathbb{F}&=\mathbb{C}\left(h_t, c_t^{\pm}, e_m; 0 \le t \le L, 1 \le m \le L-1 \right)\\ &=\mathbb{C}\left(\sqrt{\lambda^2-k_t^2}, e^{-\sqrt{\lambda^2-k_m^2}(d_{m-1}-d_m)}; 0 \le t \le L,1\le m\le L-1 \right) \end{aligned}
\end{equation}
here $h_{t}$, $c_{t}^{\pm}$, etc. are defined in \cref{def-field-gen-coeff}.
Since any coefficient of the linear systems \cref{eq-LS-sigma} is in
$\mathbb{F}$ as shown in \cref{eq-interface-cond-expanded}, it follows that
each
\begin{equation}
\sigma_{ts}^{\ast\star}(\lambda) \in\mathbb{F}.
\end{equation}

\subsection{Polynomial bound of the reflection/transmission coefficients}

An alternative point of view on the linear systems \cref{eq-LS-sigma} will
reveal the polynomial bound of the functions $\sigma_{ts}^{\ast\star}%
(\lambda)$ in a certain domain in the complex plane. This estimate will be
crucial to the error estimates on the far-field expansions.

Pick any $k_{M} \ge\max_{0 \le l \le L}k_{l}$. Pick any $T>0$. Define the open
set
\begin{equation}
\label{set-Omega-T}\Omega_{T} = \left\{  a+b\i: a>0, -aT<b<aT \right\}
\setminus(0,k_{M}]
\end{equation}
in the complex plane. Since any branch cut of $\sqrt{\lambda^{2}-k_{l}^{2}}$
is excluded from $\Omega_{T}$, $\sigma_{ts}^{\ast\star}(\lambda)$ is a
meromorphic function in $\Omega_{T}$. We claim there is a polynomial bound of
$\sigma_{ts}^{\ast\star}(\lambda)$ for $\lambda\in\Omega_{T}$ with a
sufficiently large real part. \begin{theorem}\label{theorem-poly-bound}
Suppose the function $\sigma(\lambda) \in \mathbb{F}$.
Suppose $\forall \epsilon >0$, $\sigma(\lambda) \ll \exp(\epsilon \lambda)$ as $\lambda \to +\infty$.
Then, $\exists k'_M>0$, $C>0$ and nonnegative integer $K$ such that
$|\sigma(\lambda)| \le C|\lambda|^K$
when $\lambda \in \Omega_T$ and $\Re \lambda > k'_M$.
In addition, $\sigma(\lambda)$ has finitely many poles in $\Omega_T$.
\end{theorem}
\begin{proof}
Since $\sigma(\lambda) \in \mathbb{F}$, there exist polynomials $P_1$ and $P_2$ such that
\begin{equation}
\sigma(\lambda) = \frac{I_{1}}{I_{2}} = \frac{P_{1}\left(\sqrt{\lambda^{2}-k_{l}^{2}},\cdots, e^{+\sqrt{\lambda^{2}-k_{m}^{2}}(d_{m-1}-d_m)},\cdots\right)}{P_{2}\left(\sqrt{\lambda^{2}-k_{l}^{2}},\cdots, e^{+\sqrt{\lambda^{2}-k_{m}^{2}}(d_{m-1}-d_m)},\cdots\right)}
\end{equation}
here $P_{1}$ and $P_{2}$ are polynomials of the terms in the parentheses, including terms with indices $0\le l \le L$ and $1\le m \le L-1$.
To show the asymptotic behavior of $I_1$ and $I_2$, we characterize them as elements of the ring $\mathcal{S}$ to be defined below.
Let $\Omega_{T,M} = \left\{a+b\i \in \Omega_T: a,b\in\mathbb{R}, a>k_M \right\}$ be an open subset of $\Omega_{T}$.
Let $\mathcal{G}$ be the collection of all holomorphic functions $g(\lambda)$ in $\Omega_{T,M}$ such that the number of nonzero terms with positive exponent is finite in the Laurent series of $g(\lambda)$ at $\infty$, i.e.
\begin{equation*}
\mathcal{G}=\left\{g(\lambda)=\sum_{n=0}^{\infty}c_{n}\lambda^{m-n}:m\in\mathbb{Z}\text{, }c_{n}\in\mathbb{C}\text{, } c_0 \ne 0\text{, } g(\lambda)\text{ is holomorphic in }\Omega_{T,M}\right\}.
\end{equation*}
It follows that each $\sqrt{\lambda^{2}-k_{l}^{2}}\in\mathcal{G}$, because it has neither a pole nor a branch point in $\Omega_{T,M}$ where $\Re \lambda > k_M \ge k_{l}$, and
\begin{equation}
\sqrt{\lambda^{2}-k_{l}^{2}}=\sum_{n=0}^{\infty}\frac{\sqrt{\pi}(-k_{l}^{2})^{n}}{2\Gamma(n+1)\Gamma(-n+\frac{3}{2})}\lambda^{1-2n}.
\end{equation}
Let $\mathcal{S}$ be the collection of all holomorphic functions $h(\lambda)$ in $\Omega_{T,M}$ in the form
\begin{equation}
\mathcal{S}=\left\{  h(\lambda)=\sum_{q=1}^{Q}e^{A_{q}\lambda}g_{q}(\lambda):Q\ge 0\text{, }A_{1}>\cdots>A_{Q}\geq0\text{, each }g_{q}\in\mathcal{G}\right\}.
\end{equation}
We claim that $\forall d>0$, $e^{\sqrt{\lambda^{2}-k_{i}^{2}}d}\in\mathcal{S}$.
To quickly show this fact, notice that it does not have either a pole or a branch point in $\Omega_{T,M}$, and that $e^{\sqrt{\lambda^{2}-k_{l}^{2}}d}=e^{\lambda d}e^{(\sqrt{\lambda^{2}-k_{l}^{2}}-\lambda)d}$.
For the second term, let $\mu=\lambda^{-1}$, then
\begin{equation}
e^{(\sqrt{\lambda^{2}-k_{l}^{2}}-\lambda)d} = \exp \left(\sum_{n=0}^{\infty}\frac{\sqrt{\pi}(-k_{l}^{2})^{n+1}d}{2\Gamma(n+2)\Gamma(-n+\frac{1}{2})}\mu^{2n+1} \right)
\end{equation}
which is regular in a neighborhood of $\mu=0$.
Therefore, the Laurent series in the $\mu$-plane at $0$ has zero principle part, which immediately implies $e^{(\sqrt{\lambda^{2}-k_{l}^{2}}-\lambda)d}\in\mathcal{G}$ and $e^{\sqrt{\lambda^{2}-k_{l}^{2}}d}\in\mathcal{S}$.
It is obvious that $\mathcal{G} \subset\mathcal{S}$, and $\mathcal{S}$ is a ring with function addition and multiplication.
For any function $h(\lambda) = \sum_{q=1}^{Q} e^{A_{q} \lambda} g_{q}(\lambda)\in\mathcal{G}$ which is not identical to $0$, if the leading term of $g_{1}(\lambda)$ is $B \lambda^{m}$, then
\begin{equation}
h(\lambda)\sim e^{A_{1} \lambda} B \lambda^{m}%
\end{equation}
as $\Re\lambda\to\infty$.
This is because in $\Omega_{T,M}$, $\Re \lambda \le |\lambda| \le \sqrt{1+T^2} \Re \lambda$, the limit as $|\lambda| \to \infty$ and the limit as $\Re \lambda \to \infty$ happen together.
As $|\lambda| \to \infty$, each $g_q(\lambda) \in \mathcal{G}$ approaches its leading term, in addition, as $\Re\lambda \to \infty$, $\left|e^{A_{1} \lambda} B \lambda^{m}\right|$ is larger than the sum of all the other $|e^{A_q \lambda}g_q(\lambda)|$ terms.
Now go back to $\sigma(\lambda)$.
By induction (on the total number of addition, subtraction and multiplication operations required to build up the polynomial), we have $I_{1}, I_{2} \in\mathcal{S}$.
Suppose the numerator and the denominator
\begin{equation}
I_{1} \sim e^{A_{1} \lambda} B \lambda^{m}\text{, }I_{2} \sim e^{A^{\prime}_{1} \lambda} B^{\prime}\lambda^{m^{\prime}}
\end{equation}
as $\Re\lambda\to\infty$.
Since $\sigma(\lambda) \ll \exp(\epsilon \lambda)$ as $\lambda \to +\infty$ for any $\epsilon >0$, we conclude that $A_{1} \le A^{\prime}_{1}$.
As a result, $|\sigma(\lambda)| \lesssim|\lambda|^{m-m^{\prime}}$ for $\lambda\in\Omega_{T,M}$ as $\Re\lambda\to\infty$, so the polynomial bound can be found for sufficiently large $\Re \lambda > k'_M$, and can be given in terms of $C|\lambda|^K$.
This immediately implies that poles of $\sigma(\lambda)$ in $\Omega_T$ can only be found for sufficiently small $\Re \lambda$, i.e. in a bounded region.
Hence the number of poles must be finite in $\Omega_T$.
\end{proof}


\end{document}